\documentclass[11pt,a4paper]{amsart}

\usepackage{amsthm, amsfonts, amssymb, amsmath, latexsym, enumerate, enumitem, array}
\usepackage[all]{xy}
\SelectTips{cm}{}

\usepackage{amssymb,eucal,graphicx}
\usepackage{enumerate}
\usepackage{colortbl}
\usepackage{color}
\usepackage{xcolor}
\usepackage{mathtools}

\usepackage{hyperref}
\usepackage{lineno}
 \newtheorem*{akn}{Acknowledgments}
 \newenvironment{akn*}{\begin{akn}\em}{\end{akn}}

\definecolor{darkgreen}{rgb}{.1,.7,.3}

\numberwithin{equation}{section}
\newtheorem{theo}{Theorem}[section]
\newtheorem{lemma}[theo]{Lemma}
\newtheorem{cor}[theo]{Corollary}
\newtheorem{prop}[theo]{Proposition}
\newtheorem{dfntn}[theo]{Definition}
\newtheorem{rem}[theo]{Remark}
\newtheorem{claim}[theo]{Claim}

\newtheorem{ass}[theo]{Assumptions}
\newtheorem{case1}[theo]{Case I:}
\newtheorem{case3}[theo]{Cases}
\newtheorem{case2}[theo]{Case II:}
\newtheorem{ex}[theo]{Example}
\newtheorem{fact}[theo]{FACT}
\newtheorem{facts}[theo]{FACTS}
\newenvironment{rem*}{\begin{rem}\em}{\end{rem}}
\newenvironment{ex*}{\begin{ex}\em}{\end{ex}}
\newenvironment{claim*}{\begin{claim}\em}{\end{claim}}
\newenvironment{facts*}{\begin{facts}\em}{\end{facts}}
\newenvironment{fact*}{\begin{fact}\em}{\end{fact}}
\newenvironment{case1*}{\begin{case1}\em}{\end{case1}}
\newenvironment{case2*}{\begin{case2}\em}{\end{case2}}

\newcommand{\Pin}[1]{{\mathbb P}^{#1}}

\newcommand{\Projcal}[1]{\mathbb{ P}({\mathcal #1})}

\newcommand{\scrollcal}[1]{(\Projcal{#1},\tautcal{#1})}
\newcommand {\xel} {(X, L)}

\newcommand{\tautcal}[1]{{\mathcal O}_{\mathbb{P}({\mathcal#1})}(1)}

\newcommand{\num}{\equiv}

\newcommand{\Pp}{\mathbb P}
\newcommand{\Oc}{\mathcal O}

\newcommand{\FF}{\mathbb{F}}

\newcommand{\Ee}{\mathcal{E}_e}

\newcommand{\E}{\mathcal{E}}
\newcommand{\Ext}{{\rm Ext}}

\newcommand{\cD}{\mathcal{D}}

\newcommand{\cE}{\mathcal{E}}

\newcommand{\cL}{\mathcal{L}}

\newcommand{\cF}{\mathcal{F}}

\newcommand{\cG}{\mathcal{G}}

\newcommand{\cQ}{\mathcal{Q}}
\newcommand{\cU}{\mathcal{U}}
\newcommand{\cO}{\mathcal{O}}

\makeatletter
\@namedef{subjclassname@2020}{%
  \textup{2020} Mathematics Subject Classification}
\makeatother

\allowdisplaybreaks[1]

\title[\tiny{On some ``sporadic" moduli spaces of Ulrich bundles on  some $3$-fold scrolls over ${\FF}_0$}]{On some ``sporadic" moduli spaces of Ulrich bundles on  some $3$-fold scrolls over ${\FF}_0$}

\author{Maria Lucia Fania}
\address{Maria Lucia Fania\\ Dipartimento di Ingegneria e Scienze dell'Informazione e Matematica\\Universit\`{a} degli Studi di L'Aquila\\
Via Vetoio Loc. Coppito\\67100 L'Aquila\\Italy}
\email{marialucia.fania@univaq.it}

\author{Flaminio Flamini}
\address{Flaminio Flamini\\Dipartimento di Matematica\\ Universit\`a degli Studi di Roma
Tor Vergata \\ Viale della Ricerca Scientifica, 1 - 00133 Roma\\Italy}
\email{flamini@mat.uniroma2.it}



\thanks{The first author has been supported by PRIN 2017SSNZAW. The second author has been partially supported by  the MIUR Excellence Department Project MatMod@TOV, MIUR CUP-E83C23000330006, 2023-2027,  awarded to the Department of Mathematics, University of Rome Tor Vergata.}

\begin{document}

%%%%%%%%%%%%%%%%%%%%%%%%%%%%%%%%%%%%%%%%%%%%%%%%%%%%%%%%%%%%%%%%%%%%%%%%%%%%%%
%%%%%%%%%%%%%%%%%%%% Author(s) and Address %%%%%%%%%%%%%%%%%%%%%%%%%%%%%%%%%%%
%%%%%%%%%%%%%%%%%%%%%%%%%%%%%%%%%%%%%%%%%%%%%%%%%%%%%%%%%%%%%%%%%%%%%%%%%%%%%%

%%%%%%%%%%%%%% ABSTRACT %%%%%%%%%%%%%%%%%%%%%%%%%%%%%%%%%%%%%

\maketitle

%\begin{center}
%{\it  Dedicated to the memory of Gianfranco Casnati}
%\end{center}

\begin{abstract} We investigate on the existence of some {\em sporadic} rank-$r \geqslant 1$ Ulrich vector bundles on suitable $3$-fold scrolls $X$ over the Hirzebruch surface $\mathbb{F}_0$, which arise as tautological embeddings of  projectivization of very-ample vector bundles  on $\mathbb{F}_0$ that are {\em uniform} in the sense of Brosius and Aprodu--Brinzanescu, cf. \cite{bro} and \cite{ApBr} respectively. 

Such Ulrich bundles arise as deformations of ``iterative" extensions by means of {\em sporadic} Ulrich line bundles which have been contructed in our former paper \cite{fa-fl2} (where instead higher-rank {\em sporadic} bundles were not investigated therein).  
We explicitely describe irreducible components of the corresponding {\em sporadic} moduli spaces of rank $r \geqslant 1$ vector bundles  which are Ulrich with respect to the tautological polarization on $X$. In some cases, such irreducible components turn out to be a singleton, in some other such components are generically smooth, whose positive dimension has been computed and whose general point turns out to be a slope-stable, indecomposable vector bundle. 
\end{abstract}

%%%%%%%%%%%%%%% INTRODUCTION %%%%%%%%%%%%%%%%%%%%%%%%%%%%%

\section*{Introduction} Let $X$ be a smooth, irreducible projective variety of dimension $n \geqslant 1$ and let $\mathcal O_X(H)$ be a very ample line bundle on $X$. A vector bundle $\cU$  on $X$ is said to be \emph{Ulrich with respect to $\mathcal O_X(H)$} if it satisfies suitable cohomological conditions involving some multiples of the polarization $\mathcal O_X(H)$ (cf. e.g. Definition\;\ref{def:Ulrich} below and \cite[Thm.\;2.3]{b} for equivalent conditions).  

Ulrich bundles originally appeared in Commutative Algebra, in the paper \cite{U} by B. Ulrich, as bundles enjoying suitable extremal cohomological properties. After that, attention on Ulrich bundles entered in the realm of Algebraic Geometry with the paper \cite{ESW} where, among other things, the authors compute the Chow form of a projective variety $X$ under the assumption that $X$ supports Ulrich bundles of some rank $r$.  

In recent years there has been a huge amount of work on Ulrich bundles, investigating several questions on such bundles like e.g. their existence, the minimal rank for an Ulrich vector bundle supported by the pair $(X, \mathcal O_X(H))$ (such a minimal rank is also called the {\em Ulrich complexity} of $(X, \mathcal O_X(H))$ and denoted by $uc_{\mathcal O_X(H)}(X)$), as well as their {\em stability},  their moduli space structure, etcetera (for nice surveys the reader is referred to e.g. \cite{Co,CMP}). Although a lot has been proved for some specific classes of projective varieties (e.g. curves, Segre, Veronese, Grassmann varieties, rational normal 
scrolls, hypersurfaces, some classes of surfaces and threefolds, cf. e.g. \cite{b,c-h-g-s,Co,CMP} for overviews), several questions are still open in their full  generality even for surfaces. 

In our former paper \cite{fa-fl2}, we focused on $3$-fold scrolls $X= X_e$ arising as embedding, via very-ample tautological line bundles $\Oc_{\Pp(\Ee)}(1)$, of projective bundles $\Pp(\E_e)$, where $\mathcal E_e$ are very-ample rank-$2$ vector bundles on  $\mathbb{F}_e$ with Chern classes $c_1 (\mathcal E_e)$ numerically equivalent to $3 C_e + b_ef$ and $c_2(\mathcal E_e) =k_e$, where $C_e$ and $f$ generators of ${\rm Num}(\mathbb{F}_e)$ and where $b_e$ and $k_e$ are integers satisfying some natural numerical conditions (cf. Assumptions \ref{ass:AB} below). In this set-up, one gets $3$-fold scrolls $X_e \subset \mathbb P^{n_e}$, with $n_e := 4 b_e - k_e - 6 e + 4$, which are non--degenerate, non-special, of degree $d_e:= \deg(X_e) = 6b_e - 9e - k_e$ and sectional genus $g_e:= 2b_e-3e-2$, whose hyperplane section line bundle  is denoted by $\xi_e:= \mathcal O_{X_e}(1)$ and we studied the behaviour of such $3$-fold scrolls $(X_e, \xi_e)$ in terms of some Ulrich bundles they can support. 

A reason for such interest came from the fact that the existence of Ulrich bundles on geometrically ruled surfaces has been considered in  \cite{a-c-mr,ant,cas} while  in  \cite{f-lc-pl} the existence of Ulrich bundles of rank $1$ and $2$ on  low-degree smooth $3$--fold scrolls over a surface was investigated and, among such $3$--folds, there are scrolls over $\FF_e$ with $e=0, 1$. 

In the above set-up, among other things, in \cite{fa-fl2} we proved the following:

\bigskip

\noindent
{\bf Theorem A} {\rm(\!\!\cite[{\bf Main Theorem} and {\bf Main Corollary}]{fa-fl2})} 
{\em For any integer $e \geqslant 0$, consider the Hirzebruch surface 
$\mathbb{F}_e$ and let $\Oc_{\FF_e}(\alpha,\beta)$ denote the line bundle 
$\alpha C_e + \beta f$ on $\mathbb{F}_e$, where $C_e$ and $f$ are  generators  of ${\rm Num}(\mathbb{F}_e)$. Let $(X_e, \xi_e)\cong (\Pp(\E_e), \Oc_{\Pp(\E)}(1))$ be a $3$-fold scroll over $\mathbb{F}_e$, where $\mathcal E_e$  is   as in Assumptions \ref{ass:AB} and where $\varphi: X_e \to \FF_e$ denotes the scroll map. Then: 

\medskip

\noindent 
(a) $X_e$  does  not support any Ulrich line bundle  w.r.t. the tautological polarization $\xi_e$ unless $e = 0$. In this latter case, the unique Ulrich line bundles on $X_0$ are the following: 
\begin{itemize}
\item[(i)]  $L_1 :=\xi_0+\varphi^*\Oc_{\FF_0}(2,-1) $ and $ L_2 :=\xi_0+\varphi^*\Oc_{\FF_0}(-1,b_0-1)$;  
\item[(ii)]  for any integer $t\geqslant 1$, $M_1 :=2\xi_0+\varphi^*\Oc_{\FF_0}(-1,-t-1)$ and $M_2:=\varphi^*\Oc_{\FF_0}(2,3t-1)$, which only occur for $b_0=2t, k_0=3t$.
\end{itemize}

\medskip 

\noindent
(b) Set $e=0$ and let $r \geqslant 2$ be any integer. Then the moduli space of rank-$r$ vector bundles $\cU_r$ on $X_0$ which are Ulrich w.r.t. $\xi_0$ and with first Chern class
\begin{eqnarray*}c_1(\cU_r) =
    \begin{cases}
      r \xi_0 + \varphi^*\Oc_{\FF_0}(3, b_0-3) + \varphi^*\Oc_{\FF_0}\left(\frac{r-3}{2}, \frac{(r-3)}{2}(b_0-2)\right),  & \mbox{if $r$ is odd}, \\
      r \xi_0 + \varphi^*\Oc_{\FF_0}(\frac{r}{2},\frac{r}{2}(b_0-2)), & \mbox{if $r$ is even}.
    \end{cases}\end{eqnarray*}
    is not empty and it contains a generically smooth component $\mathcal M(r)$ of dimension 
		\begin{eqnarray*}\dim (\mathcal M(r) ) = \begin{cases} \frac{(r^2 -1)}{4}(6 b_0 -4), & \mbox{if $r$ is odd}, \\
			 \frac{r^2}{4} (6b_0-4) +1 , & \mbox{if $r$ is even}.
    \end{cases}
    \end{eqnarray*}
    The general point $[\cU_r] \in \mathcal M(r)$  
corresponds to a  slope-stable vector bundle, of slope w.r.t. $\xi_0$ given by 
$\mu(\cU_r) = 8b_0 - k_0 -3$. If moreover $r=2$, then $\cU_2$ is also {\em special} (cf. Def. \ref{def:special} above).

\medskip

\noindent
(c) When $e >0$, let $r \geqslant 2$ be any integer. Then the moduli space of rank-$r$ vector bundles $\cU_{r}$ on $X_e$ which are Ulrich w.r.t. $\xi_e$ and with first Chern class
\begin{eqnarray*}c_1(\cU_r) =
    \begin{cases}
      r \xi_e + \varphi^*\Oc_{\FF_e}(3,b_e -3) + \varphi^*\Oc_{\FF_e}\left(\frac{r-3}{2}, \frac{(r-3)}{2}(b_e - e -2)\right), & \mbox{if $r$ is odd}, \\
      r \xi_e + \varphi^*\Oc_{\FF_e}\left(\frac{r}{2}, \frac{r}{2}(b_e-e-2)\right), & \mbox{if $r$ is even}.
    \end{cases}\end{eqnarray*}
    is not empty and it contains a generically smooth component $\mathcal M(r)$ of dimension 
		\begin{eqnarray*}\dim (\mathcal M(r) ) = \begin{cases} \left(\frac{(r -3)^2}{4}+ 2 \right)(6 b_e - 9e -4) + \frac{9}{2}(r-3) (2b_e-3e), & \mbox{if $r$ is odd}, \\
			 \frac{r^2}{4} (6b_e- 9e-4) +1 , & \mbox{if $r$ is even}.
    \end{cases}
    \end{eqnarray*} The general point $[\cU_{r}] \in \mathcal M(r)$  
corresponds to a  slope-stable vector bundle, of slope w.r.t. $\xi_e$ given by 
$\mu(\cU_r) = 8 b_e - k_e - 12 e - 3$.  If moreover $r=2$, then $\cU_2$ is also special.

\medskip

\noindent
(d) In particular, 
\begin{itemize}
\item[(i)]  when $e=0$, the {\em Ulrich complexity} of $X_0$ w.r.t. $\xi_0$ is 
$uc_{\xi_0}(X_0) = 1$; moreover $X_0$ supports slope-stable vector bundles of any rank $r \geqslant 1$ which are Ulrich w.r.t. $\xi_0$, i.e. there are no slope-stable-Ulrich-rank gaps on $X_0$ w.r.t. the chosen Chern class;  
\item[(ii)]  when otherwise $e>0$, the {\em Ulrich complexity} of $X_e$ w.r.t. $\xi_e$ is $uc_{\xi_e}(X_e) = 2$; nonetheless $X_e$ supports slope-stable vector bundles of any rank $r \geqslant 2$ which are Ulrich w.r.t. $\xi_e$, i.e. the only slope-stable-Ulrich-rank gap w.r.t. the chosen Chern class is $r=1$.  
\end{itemize}
}

\bigskip

Part (a) of {\bf Theorem A} above highlights in particular that, when $e=0$, $3$-fold scrolls $X_0$ support Ulrich line bundles, namely $M_1$ and $M_2$, which have a certain  {\em sporadic behavior} as they actually exist only for pairs $(b_0, k_0)= (2t, 3t)$, for any integer $t \geqslant 1$ where, we recall that the integer pair $(b_0, k_0)$ comes from 
$c_1(\mathcal E_0) = 3 C_0 + b_0 f$, $c_2(\mathcal E_0) = k_0$ (cf. also \eqref{(iii)}, \eqref{eq:rem:assAB} below) whereas line bundles $L_1$ and $L_2$ actually exist for all pairs $(b_0,k_0)$ associated to $\cE_0$ very-ample on $\mathbb F_0$. For this reason, $M_1$ and $M_2$ will be called {\em sporadic} Ulrich line bundles on $X_0$ whereas $L_1$ and $L_2$ {\em non-sporadic} Ulrich line bundles. 

We want to stress that parts (b) and (d)-(i) of {\bf Theorem A} above have been proved in \cite{fa-fl2} via {\em iterative constructions} of Ulrich vector bundles $\mathcal U_r$, of any rank $r \geqslant 2$, with the use of deformations of non-trivial extensions involving only {\em non-sporadic} Ulrich line bundles $L_1$ and $L_2$ as in part (a)-(i) of {\bf Theorem A}.

Our goal in the present paper is to focus on moduli spaces of Ulrich vector bundles on $(X_0, \xi_0)$ which arise from {\em sporadic/mixed cases}; precisely, we are interested in: 

\medskip

\noindent
$\bullet$ understanding what type of moduli spaces of rank-$r\geqslant 2$ 
vector bundles on $X_0$, which are Ulrich w.r.t. $\xi_0$, arise from iterative constructions by means of either {\em sporadic pairs} $(M_1, M_2)$ or even {\em mixed pairs} $(L_i, M_j)$, $1 \leqslant i,j \leqslant 2$, as in  {\bf Theorem A}-(a), 

\medskip 

\noindent
$\bullet$ computing what kind of Chern classes are determined by these types of constructions,

\medskip 

\noindent
$\bullet$ proving, for any $r \geqslant 2$, the existence of an irreducible component $\mathcal M(r)$ of any such a moduli space and deducing whether such a component can be (generically) smooth,

\medskip

\noindent
$\bullet$ computing $\dim(\mathcal M(r))$,

\medskip

\noindent
$\bullet$  establishing  slope-stability for the bundle $\mathcal U_r$ corresponding to a general point $[\mathcal U_r]$ of any such a component $\mathcal M(r)$, and

\medskip 

\noindent
$\bullet$ understanding whether there exists some slope-stable-Ulrich rank gap w.r.t. the chosen Chern classes.

\bigskip 

Throughout this work we will be therefore concerned with the case $e=0$, with $3$-fold scrolls arising from bundles $\mathcal E_0$ as in Assumptions \ref{ass:AB} over $\FF_0 = \Pp(\Oc^{\oplus 2}_{\Pp^1})$. Using only {\em sporadic pairs}, we prove the following:

\medskip

\noindent
{\bf Theorem B-sporadic cases} (cf. Theorem \ref{thm:general0}) {\em Let  $(X, \xi) \cong \scrollcal{E}$ be a $3$-fold  scroll over $\FF_0$, with $\mathcal E = \mathcal E_0$ satisfying Assumptions \ref{ass:AB}.  Let $\varphi: X \to \FF_0$ be the scroll map and $F$ be the $\varphi$-fiber. Let $r \geqslant 1$ be any integer. 

Then the moduli space of rank-$r$ vector bundles $\cU_r$ on $X$ which are Ulrich w.r.t. $\xi$ and with Chern classes 

\begin{equation*}
    c_1(\cU_r): =
    \begin{cases} 
      (r +1)\xi + \varphi^*\Oc_{\FF_0}(0, -2t) + \varphi^*\Oc_{\FF_0}\left(\frac{(r-3)}{2},r(t-1) \right), & \mbox{if $r$ is odd}, \\
      r \xi + \varphi^*\Oc_{\FF_0}\left(\frac{r}{2}, r(t-1)\right), & \mbox{if $r$ is even}, 
    \end{cases}
  \end{equation*}
     \begin{eqnarray*}
c_2(\cU_r) =
    \begin{cases} 
       \xi \cdot  \varphi^*\Oc_{\FF_0}\left(2r^2-2, (2t-1)r^2-2t+1 \right) -\frac{(r-1)(2rt+r+14t-3)}{2} F, & \mbox{if $r\geqslant 3$ is odd}, \\
       \xi \cdot \varphi^*\Oc_{\FF_0}\left(2r^2-2r, r(2rt-r-t+1) \right)-\frac{r(2rt+r+t-1)}{2} F, & \mbox{if $r$ is even}, 
    \end{cases}
      \end{eqnarray*}
     \begin{eqnarray*}
c_3(\cU_r) =
    \begin{cases} 
       4r^3t-2r^3-8r^2t+4r^2-4rt+2r+8t-4,   & \mbox{if $r\geqslant 3$ is odd}, \\
       4r^3t-2r^3-10r^2t+6r^2+4rt-4r, & \mbox{if $r\geqslant 4$ is even}, 
    \end{cases}
    \end{eqnarray*}is not empty and it contains a generically smooth component $\mathcal M(r)$ of dimension 
		\begin{eqnarray*}
		\dim (\mathcal M(r) ) = \begin{cases} \frac{(r^2 -1)}{4}(8t -4), & \mbox{if $r$ is odd}, \\
			 \frac{r^2}{4} (8t-4) +1 , & \mbox{if $r$ is even},
    \end{cases}
    \end{eqnarray*} with $t \geqslant 1$. The general point $[\cU_r] \in \mathcal M(r)$  
corresponds to a  slope-stable vector bundle, of slope w.r.t. $\xi$ given by 
$\mu(\cU_r) = 13t -3$. 

In particular, there are no slope-stable-Ulrich-rank gaps on $X$ w.r.t. the chosen Chern classes.}

\bigskip

When otherwise {\em mixed pairs} are considered on the one hand we show that, in some cases, there are $0$-dimensional modular components consisting only of one point which corresponds to a ($S$-equivalence class of a) polystable bundle (cf. Theorem \ref{prop:rk 2 simple Ulrich vctB e=0;II}-(1) and (4)); on the other,  using similar strategy as that used to prove {\bf Theorem B-sporadic cases}, one can get existence of several {\em extra} positive-dimensional, sporadic modular components which are different from  those in {\bf Theorem B-sporadic cases}. 

Due to the high number of possible pairings at any rank-$r$ step, we will limit here to state detailed results for some significant examples of {\em extra}, positive-dimensional sporadic modular components of Ulrich bundles on $X$, arising via the use of some specific {\em mixed pairs}. In particular we have the following:

\bigskip

\noindent
{\bf Theorem C-mixed cases} {\em Let  $(X, \xi) \cong \scrollcal{E}$ be a $3$-fold  scroll over $\FF_0$, with $\mathcal E = \mathcal E_0$ satisfying Assumptions \ref{ass:AB}.  Let $\varphi: X \to \FF_0$ be the scroll map and $F$ be the $\varphi$-fiber. Let $r \geqslant 1$ be any integer. Then $(X,\xi)$ supports several {\em extra} positive-dimensional, sporadic modular components parametrizing rank-$r$ vector bundles $\cU_r$ on $X$ which are Ulrich w.r.t. $\xi$, with given Chern classes. 

In particular, for Chern class 
$$c_1: =
    \begin{cases} 
      r \xi +\varphi^*\Oc_{\FF_0}\left(\frac{(r-3)}{2},(r-1)(t-1) +1 \right), & \mbox{if $r$ is odd}, \\
      r \xi + \varphi^*\Oc_{\FF_0}\left(\frac{r}{2}, (r-2)(t-1)\right), & \mbox{if $r$ is even},
    \end{cases}$$the moduli space of rank-$r$ vector bundles $\cU_r$ on $X$ which are Ulrich w.r.t. $\xi$ and with first Chern class as above is not empty 
    and it contains a generically smooth component $\mathcal M(r)$ of dimension 
		\begin{eqnarray*}
		\dim (\mathcal M(r) ) = \begin{cases} \frac{r^2}{4}(10t -5) + r - 2t, & \mbox{if $r$ is even}, \\
			 t(r^2 + 7r-12) - \frac{1}{2} (4r^2 - 11r +13), & \mbox{if $r$ is odd},
    \end{cases}
    \end{eqnarray*} with $t \geqslant 1$.

The general point $[\cU_r] \in \mathcal M(r)$  corresponds to a  slope-stable vector bundle, of slope w.r.t. $\xi$ given by $\mu(\cU_r) = 13t -3$. 
In particular, there are no slope-stable-Ulrich-rank gaps on $X$ w.r.t. the chosen Chern classes.}

\medskip

For a proof of the previous result, the reader is referred to Theorem \ref{prop:rk 2 simple Ulrich vctB e=0;II} and to \S\,\ref{Ulrichhighermixed}.

\medskip

The paper consists of four sections. In Section 1 we recall some generalities on 
Ulrich bundles on projective varieties as well as some other preliminaries necessary to properly define $3$-fold scrolls $(X, \xi)$ which are the core of the paper. Section \ref{Ulrich rk 2 vb} focuses   on moduli spaces of  rank-$2$ Ulrich vector bundles obtained  via both {\em sporadic} (cf. Subsect.\;\ref{S:sporadic}) and {\em mixed extensions} (cf. Subsect.\;\ref{S:mixed}). Section \;\ref{Ulrich higher rk  vb} deals with the general case of higher rank $r \geqslant 3$, obtained via 
inductive processes, extensions, deformations and modular theory dealing with {\em sporadic pairs}. Finally in Section \ref{Ulrichhighermixed} we briefly discuss some 
results which can be obtained, following same strategies as in Section \ref{Ulrich higher rk  vb}, using {\em mixed pairs}.

\begin{akn*} We would like to thank the  organizers and all participants of the two conferences, {\em ``Homemade Algebraic Geometry - Celebrating Enrique Arrondo's 60th birthday"}, held at Alcal\'a de Henares, July 10-13, 2023 and {\em ``Bandoleros 2023"}, held at Politecnico di Torino, October 23-25, 2023, not only for the perfect organizations  and for friendly and cooperative atmospheres but also because this article came about as a result of questions posed by several participants during the aforementioned conferences, regarding results on moduli spaces studied in our former paper \cite{fa-fl2}. Furthermore, we deeply thank the organizers of the conference {\em ``Bandoleros 2023"} also for the opportunity to have been in contact for the last time with our colleague and friend Gianfranco Casnati. 
\end{akn*}
%%%%%%%%%%%%%%%%%%%%%%%%%%%
%
% NOTATION 
%
%%%%%%%%%%%%%%%%%%%%%%%%%%%%

\subsection*{Notation and terminology} We work throughout over the field $\mathbb{C}$ of complex numbers. All schemes will be endowed with the Zariski topology. By \emph{variety}   we mean an integral algebraic scheme. We say that a property holds for a \emph{general}  point of a variety $V$ if it holds for any point in a Zariski open non--empty subset of $V$. We will  interchangeably use the terms {\em rank-$r$ vector bundle} on a variety $V$ and {\em rank-$r$ locally free sheaf} on $V$; in particular for 
the case $r=1$,   that   is line bundles (equiv. invertible sheaves), to ease notation and if no confusion arises, we sometimes identify line bundles with Cartier divisors interchangeably using additive notation instead of multiplicative notation and tensor products. Thus, if $L$ and $M$ are line bundles on $V$, the {\em dual} of $L$ will be 
denoted by either $L^{\vee}$, or $L^{-1}$ or even $-L$, $L^{\otimes n}$ will be sometimes denoted with $nL$ as well as $L \otimes M$ with $L+M$.
 
If $\mathcal P$ is either a {\em parameter space} of a flat family of geometric objects $\mathcal E$ defined on $V$ (e.g. vector bundles, extensions, etc.) 
or a {\em moduli space} parametrizing geometric objects modulo a given equivalence relation, we will denote by $[\mathcal E]$ the parameter point (resp., the moduli point) corresponding to the geometric object $\mathcal E$ (resp., associated to the equivalence class of $\mathcal E$). For further non-reminded terminology, we refer the reader to \cite{H}.

\vspace{3mm}
In the sequel, we will  focus on smooth, irreducible, projective $3$-folds and the following notation will be used throughout this work.

\label{notation}
  \begin{enumerate}
\item [$\bullet$ ] $X$ is a smooth, irreducible, projective variety of dimension $3$ (or simply a {\em $3$-fold});
\item [$\bullet$ ]$\chi(\mathcal F) =  \sum_{i=0}^3 (-1)^i  h^i(X, \mathcal F)$ denotes the Euler characteristic of $\mathcal F$, where $\mathcal F$ is a vector bundle of rank $r \geqslant 1$ on $X$;
\item [$\bullet$ ] $\omega_X$ denotes the canonical bundle of $X$ as well as $K_X$ denotes a canonical divisor; 
\item [$\bullet$ ] $c_i = c_i(X)$  denotes the $i^{th}$-Chern class of $X$, $0 \leqslant i \leqslant 3$;
\item [$\bullet$ ] $d = \deg{X} = L^3$ denotes the degree of $X$ in its embedding given by a very-ample line bundle $L$ on $X$;
\item [$\bullet$ ] $g = g(X),$ denotes the sectional genus of $\xel$ defined by $2g-2:=(K_X+2 L)L^2;$ 
\item [$\bullet$ ] if $S$ is a smooth surface, $\equiv$ will denote the numerical equivalence of divisors on $S$ whereas $\sim$ their linear equivalence. 
\end{enumerate}

%%%%%%%%%%%%%%%%%%%%%%%%%%%
%
% PRELIMINARIES
%
%%%%%%%%%%%%%%%%%%%%%%%%%%%%

\section{Preliminaries} For the reader convenience we recall some general facts that we will use in the sequel. 

\begin{dfntn}\label{def:Ulrich} Let $X\subset \Pp^N$ be a smooth, irreducible, projective variety of dimension $n$  and let $H$  be a hyperplane section of $X$.
A vector bundle $\cU$  on $X$ is said to be  {\em Ulrich} with respect to ${\mathcal O}_X(H)$ if
\begin{eqnarray*}
H^{i}(X, \cU(-jH))=0 \quad \text{for all}  \quad  0 \leqslant i \leqslant n \text{\, and} \quad 1 \leqslant j \leqslant n.
\end{eqnarray*}
 \end{dfntn}

  \begin{rem*}\label{Rmk1}  (i) If $X$ supports Ulrich bundles w.r.t. ${\mathcal O}_X(H)$ then one sets $uc_H(X)$, called the {\em Ulrich complexity of $X$ w.r.t. ${\mathcal O}_X(H)$}, to be the minimum rank among possible Ulrich vector bundles w.r.t. ${\mathcal O}_X(H)$ on $X$. 

\noindent
(ii) If $\cU$ is a vector bundle on $X$, which is Ulrich w.r.t. ${\mathcal O}_X(H)$,  then 
$\cU':=\cU^{\vee}(K_X +(n+1)H)$ is a vector bundle of the same rank of $\cU$, which is also Ulrich w.r.t. ${\mathcal O}_X(H)$. The vector bundle $\cU'$ is called the {\em Ulrich dual} of $\cU$. From this we see that,  if Ulrich  bundles of some given rank $r \geqslant 1$ on $X$ do exist, then they come in pairs. 
 \end{rem*}
 
\begin{dfntn}\label{def:special} Let $X\subset \Pp^N$ be a smooth, irreducible, projective variety of dimension $n$ and let $H$ denote a hyperplane section of $X$.  Let $\cU$  be  a rank-$2$ vector bundle  on  $X$ which is  {\em Ulrich} with respect to    $\mathcal O_X(H)$. Then $\cU$ is said to be {\em special} if $c_1(\cU) = K_{X} +(n + 1)H.$

 \end{dfntn}

\noindent
Notice that,  because $\cU$ in Definition \ref{def:special} is of rank $2$, then $\cU^{\vee} \cong \cU (- c_1(\cU))$; therefore  for a rank-$2$ Ulrich bundle  $\cU$   being {\em special}  is equivalent to  $\cU$  being isomorphic to its Ulrich dual bundle. 

\bigskip

We now briefly remind well-known facts  concerning (semi)stability and slope-(semi)stability properties of Ulrich bundles (cf. \cite[Def.\;2.7]{c-h-g-s}). Let $\mathcal V$ be a vector bundle on $X$; recall that $\mathcal V$ is said to be {\em semistable} if for every non-zero coherent subsheaf
$\mathcal K \subset \mathcal V$, with $0 < {\rm rk}(\mathcal K) := \mbox{rank of} \; \mathcal K < {\rm rk}(\mathcal V)$, the inequality
$\frac{P_{\mathcal K}}{{\rm rk}(\mathcal K)} \leqslant  \frac{P_{\mathcal V}}{{\rm rk}(\mathcal V)}$ holds true, where 
$P_{\mathcal K}$ and $P_{\mathcal V}$ are their Hilbert polynomials. Furthermore, $\mathcal V$ is said to be {\em stable} if 
the strict inequality above holds. 

Recall that the {\em slope} of a vector  bundle $\mathcal V$, w.r.t. the very ample polarization $\mathcal O_X(H)$, is defined to be $\mu(\mathcal V) := \frac{c_1(\mathcal V) \cdot H^{n-1}}{{\rm rk}(\mathcal V)}$; then $\mathcal V$ is said 
to be {\em $\mu$-semistable}, or even {\em slope-semistable} (w.r.t. $\mathcal O_X(H)$), if for every non-zero coherent subsheaf
$\mathcal K \subset \mathcal V$ with $0 < {\rm rk}(\mathcal K)   < {\rm rk}(\mathcal V)$, one has 
$\mu (\mathcal K) \leqslant \mu(\mathcal V)$, whereas $\mathcal V$ is said to be {\em $\mu$-stable}, or  {\em slope-stable}, if the strict inequality holds. 

The two definitions of (semi)stability are in general related as follows (cf. e.g. \cite[\S\;2]{c-h-g-s}): 
\begin{eqnarray*}\mbox{slope-stability} \Rightarrow \mbox{stability} \Rightarrow \mbox{semistability} \Rightarrow \mbox{slope-semistability}.\end{eqnarray*}

\bigskip

When the bundle in question is in particular Ulrich w.r.t. 
$\mathcal O_X(H)$, and in this case we denote it by $\mathcal U$ to remind Ulrichness, then one has   
 \begin{eqnarray} \label{slope}
\mu(\mathcal U)= \deg(X) + g-1,
\end{eqnarray} where  $\deg (X)$ is the degree of $X$ in the embedding given by $\mathcal O_X(H)$ and where $g$ is the sectional genus of $(X,\mathcal O_X(H))$ (see e.g. \cite[Prop.\,3.2.5]{CMP}), and the following more precise situation holds:

\begin{theo}\label{thm:stab} (cf. \cite[Thm.\;2.9]{c-h-g-s}) Let $X\subset \Pp^N$ be a smooth, irreducible, projective variety of dimension $n$ and let $H$ be a hyperplane section of $X$. Let $\cU$ be a rank-$r$ vector bundle on $X$ which is Ulrich w.r.t. $\mathcal O_X(H)$. Then: 

\noindent
(a) $\mathcal U$ is semistable, so also slope-semistable;

\noindent
(b) If $0 \to  \mathcal F \to \mathcal  U  \to \mathcal G \to 0$ is an exact sequence of coherent sheaves with $\mathcal G$ 
torsion-free, and $\mu(\mathcal F) = \mu(\mathcal U)$, then $\mathcal F$ and $\mathcal G$ are both vector 
bundles which are Ulrich w.r.t. $\mathcal O_X(H)$. 

\noindent
(c) If $\mathcal U$ is stable then it is also slope-stable. In particular, the notions of stability and slope-stability coincide 
for Ulrich bundles. 
\end{theo} 
\vspace{3mm}

 We like to point out that the property of being Ulrich in an irreducible, flat family of vector bundles is an open condition; indeed if the bundle $\cU$ is a deformation  of an Ulrich vector bundle  $\widetilde{\cU}$  then  $\cU$ is  also Ulrich as the cohomology vanishings  of  $\widetilde{\cU}(-j)$, for $1 \leqslant j \leqslant n$, imply (by semi--continuity in the irreducible flat family) the cohomology vanishings of $\cU(-j)$. 
\vspace{3mm}

We also like to remark  that because Ulrich bundles are semistable, then any family of Ulrich bundles with given rank and Chern classes is bounded (see for instance \cite{langer}). In this situation, if bundles in a bounded family are {\em simple}, 
i.e. $\rm End(\mathcal U) \cong \mathbb{C}$, one has:

\begin{prop} {\rm(}see \cite[Prop.\,2.10]{c-h-g-s}\label{casanellas-hartshorne}{\rm )} On a non-singular projective variety $X$, any bounded family of simple bundles $\mathcal E$ with given rank and Chern classes, satisfying $H^2(\mathcal E \otimes \mathcal E^{\vee}) = 0$ admits a {\em smooth modular family}.
\end{prop}

The existence of smooth modular families of simple vector bundles, along with the fact that the property of being Ulrich in an irreducible flat family of vector bundles is an open condition, will help us in performing {\em recursive constructions} of Ulrich bundles in any possible rank $r \geqslant 1$ on the projective varieties we are dealing with, proving also slope-stability for the general bundle parametrized by such an irreducible family.

Let us now introduce projective varieties we are interested in, which will be the support of Ulrich bundles we are going to contruct and study.

\begin{dfntn}\label{specialvar} Let $X$ be a smooth, irreducible, projective variety of dimension $3$, or simply a {\em $3$-fold}, and let $L$ be an ample line bundle on $X$.  The pair $(X,L)$ is said to be a {\em scroll} over a normal variety $Y$ if there exist an ample line bundle $M$ on $Y$ and a surjective morphism  $\varphi: X \to Y$ with
connected fibers such that $K_X + (4 - \dim(Y)) L = \varphi^*(M).$
\end{dfntn}

If, in particular, in the above definition $Y$ is a smooth surface and $\xel$ is a scroll over $Y$,  then (see \cite[Prop. 14.1.3]{BESO})
$X \cong \Projcal{E} $, where ${\mathcal E}= \varphi_{*}(L)$ is a vector bundle on $Y$ and
$L \cong \mathcal O_{\mathbb P(\mathcal E)}(1)$ is the {\em tautological  line bundle} on $\Projcal{E}.$ Moreover, if $S \in |L|$ is a smooth divisor,
then (see e.g. \cite[Thm. 11.1.2]{BESO}) $S$ turns out to be isomorphic to the blow-up of the base surface $Y$ along a reduced zero-dimensional scheme which is an element of  $c_2(\mathcal{E})$. If we denote with $d$ the {\em degree} of the pair $(X, L)$, namely the degree of the image $\Phi_L(X) 
\subset \mathbb P^n$ of $X$ via the complete linear system $|L|$, $n := h^0(X,L)-1$, then one has 
\begin{equation}\label{eq:d}
d : = L^3 = c_1^2(\mathcal{E})-c_2(\mathcal{E}).
\end{equation}

\bigskip

In this paper we will be concerned with the case in which the base surface $Y$ of the scroll  $X$, as in Definition \ref{specialvar}, is the Hirzebruch surface 
$\FF_e := \Pp(\Oc_{\Pp^1} \oplus\Oc_{\Pp^1}(-e))$, with $e \geqslant  0$ an integer. We let $\pi_e : \FF_e \to \Pp^1$ be the natural projection; it is  well known 
(cf.\,e.g.\,\cite[V, Prop.\,2.3]{H}) that ${\rm Num}(\FF_e) = \mathbb{Z}[C_e] \oplus \mathbb{Z}[f],$ where:
 $f := \pi_e^*(p)$, for any $p \in \Pp^1$, whereas
$C_e$ denotes either the unique section corresponding to the morphism of vector bundles on $\Pp^1$ $\Oc_{\Pp^1} \oplus\Oc_{\Pp^1}(-e) \to\!\!\!\to \Oc_{\Pp^1}(-e)$, when $e>0$, or the fiber of the other ruling different from that 
induced by $f$, when otherwise $e=0$. In particular$$C_e^2 = - e, \; f^2 = 0, \; C_ef = 1.$$

Let $\Ee$ be a rank-$2$ vector bundle over $\FF_e$ and let $c_i(\mathcal{E}_e)$ be its  $i^{th}$-Chern class, $0 \leqslant i \leqslant 2$. Then $c_1( \mathcal{E}_e) \num a C_e + b f$, for some $ a, b \in \mathbb Z$, and $c_2(\mathcal{E}_e) \in \mathbb Z$. For the line bundle $\cL \num \alpha C_e + \beta f$ we will also use notation $\Oc_{\FF_e}(\alpha,\beta)$. Throughout this paper, we will consider the following:

\begin{ass}\label{ass:AB} Let $e \geqslant  0$, $b_e$, $k_e$ be integers such that
\begin{equation}\label{(iii)}
b_e-e< k_e< 2b_e-4e,
\end{equation} and let ${\mathcal E}_e$ be a rank-$2$ vector bundle over $\FF_e$, with
\begin{eqnarray*}
c_1({\mathcal E}_e) \num 3 C_e + b_e f \;\; {\rm and} \;\; c_2({\mathcal E}_e) = k_e,
\end{eqnarray*} which fits in the exact sequence 
\begin{equation}\label{eq:al-be}
0 \to A_e \to {\mathcal E}_e \to B_e \to 0,
\end{equation} where $A_e$ and $B_e$ are line bundles on $\FF_e$ such that
\begin{equation}\label{eq:al-be3}
A_e \num 2 C_e + (2b_e-k_e-2e) f \;\; {\rm and} \;\; B_e \num C_e + (k_e - b_e + 2e) f
\end{equation}
\end{ass}

\noindent 
From \eqref{eq:al-be}, in particular, one has $c_1({\mathcal E}_e) = A_e + B_e \;\; {\rm and} \;\; c_2({\mathcal E}_e) = A_eB_e$.

\bigskip

As observed in \cite[Rem.\;1.8]{fa-fl2}, condition \eqref{(iii)} and the fact that 
$k_e$ must be an integer give together that 
\begin{equation}\label{eq:rem:assAB}
b_e \geqslant 3e + 2;
\end{equation} futhermore, from \eqref{eq:al-be}, any bundle $\mathcal E_e$ fitting in this exact sequence turns out to be such that $h^i(\mathcal E_e)= 0$, for any 
$i \geqslant 1$, in particular it is {\em non-special}, moreover it is {\em very ample}, i.e. $ \mathcal O_{\mathbb P(\mathcal E_e)}(1)$ is a very ample line bundle on $ \mathbb P(\mathcal E_e)$, and {\em uniform} in the sense of \cite{ApBr,bro}. 

Therefore, if we let $\varphi: \Pp(\E_e) \to  \FF_e$ to be the {\em scroll map} as in Definition \ref{specialvar}, the pair $(X,L) := (\Pp(\E_e), \Oc_{\Pp(\E)}(1))$ is a {\em $3$-fold scroll over} $\FF_e$ and $|\Oc_{\Pp(\E)}(1)|$ gives rise to an embedding
\begin{equation}\label{eq:X0}
\Phi_e:= \Phi_{|\Oc_{\Pp(\E_e)}(1)|}: \, \Pp(\E_e) \hookrightarrow  X_e \subset \Pin{n_e},
\end{equation} where $X_e := \Phi_e( \Pp(\E_e)) \subset \mathbb P^{n_e}$, $n_e := h^0(\mathcal E_e) -1 = 4b_e - k_e - 6 e + 4$, is a smooth, irreducible, non-degenerate and non-special $3$-fold scroll of degree $d_e := 6b_e - 9 e - k_e$ and sectional genus $g_e := 2b_e - 3e -2$. We 
denote by $\xi_e := \mathcal O_{X_e}(H)$, where $H$ a hyperplane section of $X_e \subset \mathbb P^{n_e}$, and call $\xi_e$ the {\em tautological polarization} of $X_e$, as 
$(X_e, \xi_e) \cong (\Pp(\E_e), \Oc_{\Pp(\E)}(1))$. 

\medskip

In this set-up, in \cite{fa-fl2} we proved {\bf Theorem A} as in Introduction; as already mentioned therein, {\bf Theorem A}-(a) highlights that, when $e=0$, $3$-fold scrolls $X_0$ support {\em sporadic} Ulrich line bundles $M_1$ and $M_2$, which actually exist only for $(b_0, k_0)= (2t, 3t)$, for $t \geqslant 1$ an integer, where $(b_0, k_0)$ are s.t. $c_1(\mathcal E_0) = 3 C_0 + b_0 f$, $c_2(\mathcal E_0) = k_0$, with $b_0,\,k_0$ as in \eqref{(iii)}, \eqref{eq:rem:assAB}. We want to stress that part (b) and (d)-(i) of {\bf Theorem A} have been proved in \cite{fa-fl2} using {\em iterative constructions} of vector bundles $\mathcal U_r$ of any rank $r \geqslant 2$, which are Ulrich w.r.t. $\xi_0$, obtained via deformations, extensions and {\em modular families} (as in Proposition \ref{casanellas-hartshorne}) arising from {\em non-sporadic} line bundles $L_1$ and $L_2$ as in part (a)-(i) of {\bf Theorem A}.

\medskip

The present paper is therefore focused  on {\em sporadic cases} as well as on {\em mixed cases}; as explained in the Introduction.

\bigskip 

Thus, throughout this work we will be concerned with the case $e=0$ and with $3$-fold scrolls arising from bundles $\mathcal E_0$ as in Assumptions \ref{ass:AB} over $\FF_0 = \Pp(\Oc^{\oplus 2}_{\Pp^1})$. To ease notation, taking into account \eqref{eq:rem:assAB},  from now on we will simply set 
\begin{equation}\label{eq:notation0}
\E := \E_0, \; X:= X_0, \; \xi:= \xi_0,\; b:= b_0 = 2t,\; k:= k_0= 3t,
\end{equation} for any integer $t \geqslant 1$, so that $X \subset \mathbb P^{n}$, where $n := n_0 = 5t+4$, is a smooth, irreducible, non-special and linearly normal $3$-fold scroll over $\FF_0$, of degree $d:= d_0 = 9t$, sectional genus $g := g_0 = 4t -2$ and whose tautological 
polarization is $\xi = \mathcal O_X(1)$.

\bigskip

%%%%%%%%%%%%%%%%%%%%%%%%%%%%%%%%%%%%%%%%%%%%%%
%
%  	NONN -TRIVIAL EXTENSIONS M_i and M_j (RANK 2 Ulrich on F_0)
%
%%%%%%%%%%%%%%%%%%%%%%%%%%%%%%%%%%%%%%%%%%%%%
\section{Rank-$2$ Ulrich bundles arising from sporadic and mixed extensions} \label{Ulrich rk 2 vb}

From \eqref{eq:notation0}, consider any $3$-fold scroll $(X, \xi)$, $X \subset \mathbb P^n$, arising from $\cE$ satisfying Assumptions \ref{ass:AB}. The fiber of the natural scroll map $\varphi: X \cong \Pp(\E) \to \mathbb{F}_0$ will always be denoted by $F$. Take further line bundles as in {\bf Theorem A}-(a).

\bigskip
%%%%%%%%%%%%%%%%%%%
%. CASO M_i-M_j
%
%%%%%%%%%%%%%%%%%%%
Let us consider first the two {\em sporadic} Ulrich line bundles $M_1$ and $M_2$ on $X$ of {\bf Theorem A}-(a-ii), which only occur for $(b,k) = (2t,3t)$, where $t \geqslant 1$ any integer.  From  
\cite[\S\,3.1-Case M]{fa-fl2} we know  that  
 $$\dim({\rm Ext}^1(M_2,M_1))=6t-3 \geqslant 3.$$Thus there are non-trivial  extensions
 \begin{eqnarray}
\label{extension1}
0 \to M_1  \to \cF \to M_2 \to 0
\end{eqnarray} of  $M_2$ by $M_1$, i.e. $\cF$ is a rank-$2$ vector bundle on $X$ which 
corresponds to a non-zero vector of the vector space  ${\rm Ext}^1(M_2,M_1)$. 
Similarly, $$\dim({\rm Ext}^1(M_1,M_2))=2t+1 \geqslant 3,$$ so there are also non-trivial extensions 
\begin{eqnarray}
\label{extension1'}
0 \to M_2  \to \cF' \to M_1 \to 0
\end{eqnarray} of  $M_1$ by $M_2$. 

Observe that the dimensions of the above family of extensions are different positive integers unless $t=1$. We like to point out that in such a case, i.e. when $t=1$, then  $b=2$, i.e. $c_1(\E)=3C_0+2f$, so being  $\E$ very-ample with $c_1(\E)^2=12$, by  \cite[Cor.\,2.6-(ii)]{Is}  it  follows that $\E=\Oc_{\FF_0}(1,1)\oplus\Oc_{\FF_0}(1,2)$. Moreover  since on $\FF_0$ one can exchange the  two distinct rulings $C_0$ and $f$, then $c_1(\E)=3C_0+2f =2C_0+3f$, is the same.  In  such  case $X \subset \Pp^9$ has degree $9$, sectional genus $2$ and it is isomorphic to $\Pp^1 \times \FF_1$  (see \cite[(3.12), (3.4)]{Io}, or  \cite[Prop. (3.1)]{fa-li}).

Turning back to the general case, notice that the vector bundles $\cF$ and $\cF'$ are both rank-$2$  vector bundles which are Ulrich w.r.t. $\xi$, such that  
\begin{eqnarray}\label{chernF1} 
\begin{array}{ccccl}
c_1(\cF) \, = \, c_1(\cF')\, &= &  c_1(M_1) + c_1(M_2) & = & 2\xi+\varphi^*\Oc_{\FF_0}(1,2t-2),\\
c_2(\cF)\, = \, c_2(\cF') \, & = & c_1(M_1) \cdot c_1(M_2) &  = &  \xi \cdot \varphi^*\Oc_{\FF_0}(4,6t-2)- (5t+1)F.
\end{array}
\end{eqnarray} Moreover, by \eqref{slope},  one has 
\begin{eqnarray}\label{slopeMi} 
\mu(M_1) = \mu(M_2) = \mu(\cF) =\mu(\cF') = d + g - 1 = 13t-3. 
\end{eqnarray} From Theorem \ref{thm:stab}, $\cF$ and $\cF'$ are both strictly semistable, so they are $S$-equivalent (in the GIT sense,\,c.f.\,e.g. \cite[p.\,1250083-9]{c-h-g-s} and \cite[Rem.\;3.2,\, Claim\,3.3]{fa-fl2}) to $M_1 \oplus M_2$, i.e. they give rise to a point in the moduli space $\mathcal M^{ss}(2; c_1,\;c_2)$, where 
$$c_1:= 2\xi+\varphi^*\Oc_{\FF_0}(1,2t-2),\; c_2:=  \xi \cdot \varphi^*\Oc_{\FF_0}(4,6t-2)- (5t+1)F,$$parametrizing ($S$-equivalence classes of) rank-$2$  semistable sheaves with the given Chern classes. %Since $M_1$ and $M_2$ are also non-isomorphic and with the same slope w.r.t. $\xi$ then, by \cite[Lemma 4.2]{c-h-g-s}, $\cF_1$ and $\cF_1'$ are simple vector bundles, in particular they are indecomposable.  

\vspace{2mm}
%%%%%%%%%%%%%%%%%%%
%. CASO L-M
%
%%%%%%%%%%%%%%%%%%%
If instead one considers {\em mixed extensions}, namely using both line bundles of type $L_i$ as in {\bf Theorem A}-(a-i), $1 \leqslant i \leqslant 2$, (which are {\em non-sporadic} on $X$ as they exist for any pair $(b,k)$ satisfying \eqref{(iii)} and \eqref{eq:rem:assAB}) and line-bundles of type $M_j$, with $1 \leqslant j \leqslant 2$, one has non--trivial extensions only in the following cases (cf. computations as in \cite[\S\,3.1]{fa-fl2}):
\begin{case3} \label{cases} {\normalfont Non-zero extension spaces are: 
\begin{itemize}
\item[(1)] ${\rm Ext}^1(L_1,M_1)$,  where $\dim({\rm Ext}^1(L_1,M_1))=1$ and non-trivial extensions $\cF_1$   are
such that 
\[
\begin{array}{ccccl}
c_1(\cF_1) & = & c_1(L_1) + c_1(M_1) & = & 3\xi+\varphi^*\Oc_{\FF_0}(1,-t-2), \\
c_2(\cF_1) & = & c_1(L_1) \cdot c_1(M_1) & = & \xi \cdot \varphi^*\Oc_{\FF_0}(9,3t-3)- (8t+1)F;
\end{array}
\]
\item[(2)] \label{M2L1} ${\rm Ext}^1(M_2,L_1)$, where \, $\dim({\rm Ext}^1(M_2,L_1))= 10t-5 \geqslant 5$ and    non-trivial extensions are such that
\[
\begin{array}{ccccl}
c_1(\cF_2) & = & c_1(L_1) + c_1(M_2) & = & \xi+\varphi^*\Oc_{\FF_0}(4,3t-2)), \\
c_2(\cF_2) & = & c_1(L_1) \cdot c_1(M_2) & = & \xi \cdot \varphi^*\Oc_{\FF_0}(2,3t-1)+ (6t-4)F;
\end{array}
\]
\item[(3)]  ${\rm Ext}^1(L_2,M_1)$,  where $\dim({\rm Ext}^1(L_2,M_1))=10t-5 \geqslant 5$  and non-trivial  extensions $\cF_3$ are such that 
\[
\begin{array}{ccccl}
c_1(\cF_3) & = & c_1(L_2) + c_1(M_1) & = & 3\xi+\varphi^*\Oc_{\FF_0}(-2,t-2), \\
c_2(\cF_3) & = & c_1(L_2) \cdot c_1(M_1) & = & \xi \cdot \varphi^*\Oc_{\FF_0}(3,7t-3)+ (2-7t)F;
\end{array}
\]
\item[(4)] ${\rm Ext}^1(M_2,L_2)$, where  $\dim({\rm Ext}^1(M_2,L_2))=1$ and non-trivial extensions $\cF_4$ with 
\[
\begin{array}{ccccl}
c_1(\cF_4) & = & c_1(L_2) + c_1(M_2) & = & 3\xi+\varphi^*\Oc_{\FF_0}(1,5t-2), \\
c_2(\cF_4) & = & c_1(L_2) \cdot c_1(M_2) & = & \xi \varphi^*\Oc_{\FF_0}(2,3t-1)+ (t-1)F.
\end{array}
\]
\end{itemize}}
\end{case3} Reasoning as for {\em sporadic} extensions \eqref{extension1} and \eqref{extension1'} and looking at Chern-classes computations above, one determines different moduli spaces $\mathcal M^{ss}(2; c_1,\;c_2)$, according to the chosen Chern classes $c_1$ and $c_2$.

\subsection{Rank-two Ulrich bundles arising from sporadic extensions}\label{S:sporadic} Here we focus on {\em sporadic} extensions \eqref{extension1} and \eqref{extension1'}, i.e. extensions arising from the two {\em sporadic} line bundles $M_1$ and $M_2$ as in {\bf Theorem A}-(a-ii). We prove the following result.

%%%%%%%%%%%
%
%. TEOREMA  2
%
%%%%%%%%%%%

\begin{theo}\label{prop:rk 2 simple Ulrich vctB e=0;I}  Let  $(X, \xi) \cong \scrollcal{E}$ be a $3$-fold  scroll over $\FF_0$, with  $\mathcal E = \mathcal E_0$ as in Assumptions \ref{ass:AB}.  Let $\varphi: X \to \FF_0$ be the scroll map and $F$ be the $\varphi$-fiber. Then, for any $t \geqslant 1$, the moduli space of rank-$2$ vector bundles $\cU$ on $X$,  which are Ulrich w.r.t. $\xi$ and  with Chern classes 
\begin{equation}\label{eq:chernM}
c_1(\cU)=2\xi+\varphi^*\Oc_{\FF_0}(1,2t-2) \;\; {\rm and} \;\; c_2(\cU)=\xi \cdot \varphi^*\Oc_{\FF_0}(4,6t-2)- (5t+1)F, 
\end{equation} is not empty and it contains a generically smooth irreducible component $\mathcal M := \mathcal M(2)$, of dimension 
\begin{eqnarray*}\dim (\mathcal M) = 8t-3, 
\end{eqnarray*}
whose general point  $[\cU] \in \mathcal M$ 
corresponds to a  special and slope-stable vector bundle, of slope 
$\mu(\cU) =13t - 3$ w.r.t. $\xi$.

Furthermore, deformations of Ulrich bundles $\cF$ and $\cF'$ as in \eqref{extension1} and \eqref{extension1'} belong to the aforementioned component $\mathcal M$. 
\end{theo}

\begin{proof}  The proof is similar to that of \cite[Thm.\,3.1]{fa-fl2}. For the reader's convenience we will recall here main arguments for the proof.

We consider  bundles $\cF$ arising from non--trivial extensions \eqref{extension1}, where $\dim ({\rm Ext}^1(M_2,M_1))=6t-3\geqslant 3$. Any such $\cF$ 
is Ulrich w.r.t. $\xi$, thus, from \eqref{slope}, one has $\mu(\cF) = 13t-3$.  

Since, for the same reason, $\mu(M_1) = \mu(M_2) = 13t - 3$ and since furthermore $M_1$ and $M_2$ are both slope--stable, of the same slope w.r.t. $\xi$ and non--isomorphic, by \cite[Lemma\,4.2]{c-h-g-s}, any such bundle $\cF$ is simple, i.e. $h^0(\cF\otimes \cF^{\vee}) = 1$, in particular it is indecomposable.

With the use of \eqref{extension1} and its dual sequence, one can easily show that  $h^2(\cF \otimes\cF^{\vee}) = 0 = h^3(\cF \otimes\cF^{\vee})$ and that 
$\chi(\cF \otimes\cF^{\vee})=-8t+4$. Indeed, tensoring \eqref{extension1} with $\cF^{\vee}$ one gets
\begin{eqnarray}
\label{extension1tensorFdual}
0 \to M_1 \otimes \cF^{\vee}   \to \cF\otimes \cF^{\vee} \to  M_2 \otimes \cF^{\vee} \to 0;
\end{eqnarray} moreover, dualizing \eqref{extension1} gives 
\begin{eqnarray}\label{extension1dual}
0 \to M_2^{\vee}  \to \cF^{\vee} \to M_1^{\vee}\ \to 0
\end{eqnarray} which, if tensored by $M_1$ and $M_2$, respectively, gives
\begin{eqnarray}
\label{extension1dualM1}
0 \to M_2^{\vee}\otimes M_1(=2\xi+ \varphi^*\Oc_{\FF_0}(-3,-4t)) \to M_1 \otimes \cF^{\vee} \to {\cO}_{X} \to 0
\end{eqnarray} 
\begin{eqnarray}
\label{extension1dualM2}
0 \to  {\cO}_{X} \to M_2 \otimes \cF^{\vee} \to M_2 \otimes M_1^{\vee}(=-2\xi + \varphi^*\Oc_{\FF_0}(3,4t)) \to 0
\end{eqnarray} 

Because  $\cF$ is simple, then $h^0(X,  \cF\otimes \cF^{\vee} )=1$; the other cohomology vector-spaces \linebreak $H^i(X,  \cF\otimes \cF^{\vee} )$, $1 \leqslant i \leqslant 3$,  can be easily computed from the cohomology sequence associated with  \eqref{extension1dualM1} and \eqref{extension1dualM2}. Indeed, $h^i(\cO_{X})=0$ if $i\geqslant 1$ and  $h^0(\cO_{X})=1$, whereas 
%%%%%%%%%%%%%%%%%%%%%%%%%%%%%%%%%%%%%%%%%%%%%
%
% CALCOLO DI
%  $H^i(X,-2\xi+\varphi^*\Oc_{\FF_0}(3,4t))$.
%
%%%%%%%%%%%%%%%%%%%%%%%%%%%%%%%%%%%%%%%%%%%%%
\begin{eqnarray}\label{eq:calcoliutili1}
\qquad  H^i(X, -2\xi+\varphi^*\Oc_{\FF_0}(3,4t))&\cong& H^{3-i}(X,\varphi^*\Oc_{\FF_0}(-2,-2-2t))\\
 &\cong& H^{3-i}(\FF_{0}, \Oc_{\FF_{0}}(-2,-2-2t))\nonumber \\
 &\cong& H^{i-1}(\Pp^1, \Oc_{\Pp^1}(2t))=\left\{
    \begin{array}{ll}
0 &\mbox{ if $i=0,2,3$}\\
     2t+1&   \mbox{ if  $i=1$}
    \end{array}
\right.\nonumber
 \end{eqnarray}
%%%%%%%%%%%%%%%%%%%%%%%%%%%%%%%%%%%%%%%%%%%%%
%
% CALCOLO DI
%  $H^i(X,2\xi+\varphi^*\Oc_{\FF_0}(-3,-4t))$.
%
%%%%%%%%%%%%%%%%%%%%%%%%%%%%%%%%%%%%%%%%%%%%%
and, as  it was done in \cite[{\bf Case M}]{fa-fl2} 
\begin{eqnarray}\label{eq:calcoliutili2}
H^i( 2\xi+ \varphi^*\Oc_{\FF_0}(-3,-4t)&=&\left\{
    \begin{array}{ll}
0 &\mbox{ if $i=0,2,3$}\\
     6t-3&   \mbox{ if  $i=1$}
    \end{array}
\right.
 \end{eqnarray}
 
From the previous computations and \eqref{extension1tensorFdual}, it thus follows that $h^2(\cF \otimes\cF^{\vee}) = 0 = h^3(\cF \otimes\cF^{\vee})$. Once again from \eqref{extension1tensorFdual}, one has that 
\begin{eqnarray*}\chi(\cF \otimes \cF^{\vee} )=\chi( M_1 \otimes \cF^{\vee})+\chi(M_2 \otimes \cF^{\vee})=-8t+4,
\end{eqnarray*}
and  thus, from the previous vanishings and from simplicity of $\cF_1$, one finds  $h^1(\cF \otimes \cF^{\vee} )= 1 - \chi(\cF \otimes \cF^{\vee} ) = 8t-3$.

Simplicity of $\cF$ and $h^2(\cF\otimes \cF^{\vee}) = 0$ give, by Proposition \ref{casanellas-hartshorne} (cf.\,also\,\cite[Prop.\,2.10]{c-h-g-s}), that there exists a smooth modular family for $\cF$. Furthermore, since $\cF$ is Ulrich w.r.t. $\xi$, with Chern classes as in  \eqref{chernF1}, i.e. 
$c_1(\cF)=2\xi+\varphi^*\Oc_{\FF_0}(1,2t-2)$ and $c_2(\cF)=\xi \cdot \varphi^*\Oc_{\FF_0}(4,6t-2)- (5t+1)F$, the general element $\cU$ of the smooth modular family to which $\cF$ belongs corresponds to a rank-$2$ vector bundle, with same Chern classes (namely as those of $\cF$), which is Ulrich w.r.t. $\xi$, as it follows from the facts that, in irreducible flat families, Ulrichness is an open condition (by semi-continuity) and Chern classes are invariants.

Finally, one shows that $\cU$ is also slope--stable w.r.t. $\xi$. Indeed, by Theorem \ref{thm:stab}--(b) (cf. also \cite[\S\,3,\,(3.2)]{b}), if $ \cU $ were not a stable bundle, being Ulrich it would be presented as an extension of Ulrich line bundles on $X$. In such a case, by the classification of Ulrich line bundles on $X$ given in {\bf Theorem A}-(a), by Chern classes reasons we see that the only possibilities for $\cU$ to arise as an extension of Ulrich line bundles should be either  extensions  \eqref{extension1} or extensions \eqref{extension1'}. In both cases the dimension of (the projectivization) of the corresponding extension space is either $6t - 4$ or $2t$. On the other hand, by semi-continuity on the smooth modular family, 
one has 
\begin{eqnarray*} h^j(\cU\otimes \cU^{\vee}) = 0 = h^j(\cF\otimes \cF^{\vee}), \; 2 \leqslant j \leqslant 3, \; {\rm and} \; h^0(\cU\otimes \cU^{\vee}) = 1 = h^0(\cF\otimes \cF^{\vee}),
\end{eqnarray*}
thus
 \begin{eqnarray*}h^1(\cU\otimes \cU^{\vee}) = 1 - \chi (\cU\otimes \cU^{\vee} ) = 1 - \chi(\cF\otimes \cF^{\vee}) = h^1(\cF\otimes \cF^{\vee}) =8t-3,
 \end{eqnarray*}
 as computed above.  In other words, 
the smooth modular family whose general element is $\cU$ is of dimension $8t - 3$, which is bigger than both $6t-4$ and $2t$,  for any $t \geqslant 1$. This shows 
that $\cU$ general in the smooth modular family corresponds to a stable, and so also slope-stable bundle  (cf. Theorem \ref{thm:stab}-(c) above).

By slope-stability of $\cU$, we deduce that the moduli space of rank-$2$ Ulrich bundles with Chern classes as in  \eqref{chernF1} is not empty and 
it contains an irreducible component $\mathcal M = \mathcal M(2)$ where 
$[\cU] \in \mathcal M$ is a smooth point, as $h^2(\cU\otimes \cU^{\vee})=0$. Thus, 
$\mathcal M$ is generically smooth, of dimension $h^1(\cU\otimes \cU^{\vee}) = 8t-3$, from which one also deduces that $[\mathcal U]$ is a general point in $\mathcal M$. Moreover, being Ulrich from \eqref{slope} one gets $\mu(\cU) = 13t -3$. 

Note further that $\cU^{\vee} \cong \cU(- c_1(\cU))$, as $\cU$ is of rank $2$, and that 
\begin{eqnarray*}K_{X}+4\xi&=&-2\xi+\varphi^*\Oc_{\FF_0}(1,2t-2)+4\xi=2\xi+\varphi^*\Oc_{\FF_0}(1,2t-2)\\&=&c_1(\cF) = c_1(\cU)
\end{eqnarray*} 
and thus 
\begin{eqnarray*}
\cU^{\vee} \cong \cU(- c_1(\cU)) = \cU (-K_{X} - 4 \xi),
\end{eqnarray*}  i.e. that $\cU^{\vee}(K_{X} + 4 \xi) \cong \cU$ in other words 
$\cU$ is isomorphic to its Ulrich dual bundle, that is  $\cU $ is special, as stated.

To prove the last part of the statement, one uses same arguments as in \cite[Claim\,3.3]{fa-fl2}. Namely one shows that $M_1 \oplus M_2$ is a smooth point of the Quot-scheme parametrizing simple bundles with given Hilbert polynomal, so there exists a unique irreducible component $\mathcal R$ of such a Quot scheme, containing therefore also bundles $\cF$ and $\cF'$ as in \eqref{extension1} and \eqref{extension1'} and all their deformations, in particular containing $\mathcal U$. This component $\mathcal R$ 
then projects, via GIT quotient, onto the modular component $\mathcal M$ as in the statement.  
\end{proof}

%%%%%%%%%%%%%%%%%%%%%%
%
%. TEOREMA  3
%
%  EXTENSIONS   \begin{eqnarray}
%0 \to M_1  \to \cF_1 \to M_2 \to 0
%
%%%%%%%%%%%%%%%%%%%%%%

\subsection{Rank-two Ulrich bundles arising from mixed extensions}\label{S:mixed} Here we consider instead {\em mixed extensions} as in Cases \ref{cases} above.  One has the following result.

\begin{theo}\label{prop:rk 2 simple Ulrich vctB e=0;II}  Let  $(X, \xi) \cong \scrollcal{E}$ be a $3$-fold  scroll over $\FF_0$, with  $\mathcal E = \mathcal E_0$ as in Assumptions \ref{ass:AB}.  Let $\varphi: X \to \FF_0$ be the scroll map and $F$ be the $\varphi$-fiber. Then the moduli spaces of rank-$2$ vector bundles $\cU$ on $X$ which are Ulrich w.r.t. $\xi$ and  with Chern classes, respectively, 

\begin{itemize}
\item[(1)] $c_1(\cU)=3\xi+\varphi^*\Oc_{\FF_0}(1,-t-2) \; {\rm and} \; c_2(\cU)=\xi \cdot \varphi^*\Oc_{\FF_0}(9,3t-3)- (8t+1)F$, 
\vspace{1mm}
\item[(2)] $c_1(\cU)=\xi+\varphi^*\Oc_{\FF_0}(4,3t-2)) \; {\rm and} \; c_2(\cU)=\xi\cdot  \varphi^*\Oc_{\FF_0}(2,3t-1)+ (6t-4)F$, 
\vspace{1mm}
\item[(3)]  $c_1(\cU)=3\xi+\varphi^*\Oc_{\FF_0}(-2,t-2) \; {\rm and} \; c_2(\cU)=\xi \cdot  \varphi^*\Oc_{\FF_0}(3,7t-3)+ (2-7t)F$, 
\vspace{1mm}
\item[(4)]  $c_1(\cU)=\xi+\varphi^*\Oc_{\FF_0}(1,5t-2) \; {\rm and} \; c_2(\cU)=\xi \cdot \varphi^*\Oc_{\FF_0}(2,3t-1)+ (t-1)F$
\end{itemize} are not empty. Each of them contains a generically smooth irreducible component $\mathcal M:= \mathcal M(2)$ of dimension, respectively,  
\begin{eqnarray*}\dim (\mathcal M) =  \begin{cases}
      0, & \mbox{in case $(1)$}, \\
      10t-6, & \mbox{in case $(2)$},\\   
      10t-6, & \mbox{in case $(3)$},\\
      0, &  \mbox{in case $(4)$},
\end{cases}
\end{eqnarray*}

In cases $(1)$ and $(4)$, $\mathcal M$ consists of only one point represented by a polystable bundle (more precisely, in case $(1)$ one has $\mathcal M = \{[L_1 \oplus M_1]\}$ whereas in case $(4)$ one has  $\mathcal M = \{[L_2 \oplus M_2]\}$). In cases 
$(2)$ and $(3)$ the general point  $[\cU] \in \mathcal M$ corresponds to a  special and slope-stable vector bundle, of slope $\mu(\cU) =13t - 3$ w.r.t. $\xi$. 
\end{theo}

\begin{proof}  The proof goes similarly as that of Theorem \ref{prop:rk 2 simple Ulrich vctB e=0;I}. The main difference resides in the fact that computations are performed by using extension bundles as in Cases \ref{cases}. 

In case $(1)$, we take into account bundles $\cF_1$ arising from non-trivial extensions in ${\rm Ext}^1(L_1,M_1)$, where $\dim ({\rm Ext}^1(L_1,M_1))=1$. From 
\cite[Lemma\,4.2]{c-h-g-s}, $\cF_1$ is simple so  $h^0(\cF_1\otimes \cF_1^{\vee} )=1$. The remaining cohomologies $h^i(\cF_1\otimes \cF_1^{\vee} )$, $1 \leqslant i \leqslant 3$,  can be easily computed as follows: tensoring with $\cF_1^{\vee}$ the exact sequence defining $\cF_1$, one gets
\begin{eqnarray}\label{seq1}
0 \to M_1 \otimes \cF_1^{\vee}   \to \cF_1\otimes \cF_1^{\vee} \to  L_1 \otimes \cF_1^{\vee} \to 0;
\end{eqnarray} taking moreover the dual sequence of that defining $\cF_1$ and tensoring it with, respectively, $M_1$ and $L_1$ gives
\begin{eqnarray}
\label{extension1dualL1}
0 \to M_1\otimes L_1^{\vee}(=\xi+ \varphi^*\Oc_{\FF_0}(-3,-t)) \to M_1 \otimes \cF_1^{\vee} \to {\cO}_{X} \to 0
\end{eqnarray} 
\begin{eqnarray}
\label{extension1bdualM2}
0 \to  {\cO}_{X} \to L_1 \otimes \cF_1^{\vee} \to L_1 \otimes M_1^{\vee}(=-\xi + \varphi^*\Oc_{\FF_0}(3,t)) \to 0. 
\end{eqnarray} Clearly $h^i(\cO_{X})=0$ if $i\geqslant 1$ and  $h^0(\cO_{X})=1$; it remains to compute 
$h^i(X, \xi+ \varphi^*\Oc_{\FF_0}(-3,-t))$ and $h^i(X, -\xi + \varphi^*\Oc_{\FF_0}(3,t))$, for $i \geqslant 0$.  Notice that 
$$H^i(X, -\xi + \varphi^*\Oc_{\FF_0}(3,t))\cong  H^{i}(\FF_0,0) =0 \;\;  \mbox{for $i \geqslant 0$},$$whereas
$$H^i(X,\xi+ \varphi^*\Oc_{\FF_0}(-3,-t))\cong H^i( \E\otimes\Oc_{\FF_0}(-3,-t)).$$
From \eqref{eq:al-be} and \eqref{eq:al-be3} we have that $\E$ fits in
$$0 \to 2 C_0 + t f  \to {\mathcal E} \to C_0 + t f \to 0.$$If we let  $\E(-3,-t):=\E\otimes \Oc_{\FF_0}(-3,-t)$ and if we  tensor this exact sequence with $\Oc_{\FF_0}(-3,-t)$, we get 
\begin{equation}\label{eq:aEbTensor(-3,-t)}
0 \to - C_0  \to {\mathcal E}(-3,-t) \to -2 C_0 \to 0.
\end{equation}From the cohomology sequence associated to \eqref{eq:aEbTensor(-3,-t)} it follows therefore that  
\begin{eqnarray*}
h^i( X,\xi+\varphi^*\Oc_{\FF_0}(-3,-t)) = h^i(\FF_0,\E(-3,-t))=\left\{
    \begin{array}{ll}
0 &\mbox{ if $i=0,2,3$}\\
    1&   \mbox{ if  $i=1$}
    \end{array}
\right.
 \end{eqnarray*}Thus, from \eqref{extension2dualL1}, \eqref{extension2dualM2} and \eqref{seq2}, it follows that $h^2(\cF_1\otimes\cF_1^{\vee}) = 0 = h^3(\cF_1\otimes\cF_1^{\vee})$ and
\begin{eqnarray*}\chi(\cF_1\otimes \cF_1^{\vee} )=\chi( L_1 \otimes \cF_1^{\vee})+\chi(M_1 \otimes \cF_1^{\vee})=1+0=1,
\end{eqnarray*} so, from simplicity of $\cF_1$ and the previous vanishings, one has $h^1(\cF_1\otimes \cF_1^{\vee} )= 1 - \chi(\cF_1\otimes \cF_1^{\vee} ) =0$. Since any bundle $\cF_1$ arising from non-trivial extensions in  ${\rm Ext}^1(L_1,M_1)$ is strictly-semistable, from $S$-equivalence of such bundles, we get the statement.

\smallskip

In case $(2)$, we consider bundles $\cF_2$ arising from non-trivial extensions in ${\rm Ext}^1(M_2,L_1)$, where $\dim( {\rm Ext}^1(M_2,L_1)) =10t-5\geqslant 5$. From 
\cite[Lemma\,4.2]{c-h-g-s}, one gets that  $\cF_2$ is simple, i.e. $h^0(\cF_2\otimes \cF_2^{\vee} )=1$. The remaining cohomologies $h^i(\cF_2\otimes \cF_2^{\vee} )$, $1 \leqslant i \leqslant 3$,  can be easily computed as above. We tensor by $\cF_2^{\vee}$ the exact sequence defining $\cF_2$ to get 
\begin{eqnarray}\label{seq2}
0 \to L_1 \otimes \cF_2^{\vee}   \to \cF_2\otimes \cF_2^{\vee} \to  M_2 \otimes \cF_2^{\vee} \to 0.
\end{eqnarray} Moreover, taking the dual sequence of that defining $\cF_2$ and tensoring it with, respectively, $L_1$ and $M_2$ gives
\begin{eqnarray}
\label{extension2dualL1}
0 \to M_2^{\vee}\otimes L_1(=\xi+ \varphi^*\Oc_{\FF_0}(0,-3t)) \to L_1 \otimes \cF_2^{\vee} \to {\cO}_{X} \to 0
\end{eqnarray} 
\begin{eqnarray}
\label{extension2dualM2}
0 \to  {\cO}_{X} \to M_2 \otimes \cF_2^{\vee} \to M_2 \otimes L_1^{\vee}(=-\xi + \varphi^*\Oc_{\FF_0}(0,3t)) \to 0
\end{eqnarray} Clearly $h^i(\cO_{X})=0$ if $i\geqslant 1$ and  $h^0(\cO_{X})=1$, so we need to compute  
$h^i(X, \xi+ \varphi^*\Oc_{\FF_0}(0,-3t))$ and $h^i(X, -\xi + \varphi^*\Oc_{\FF_0}(0,3t)).$ Notice that 
$$H^i(X, -\xi+\varphi^*\Oc_{\FF_0}(0,3t))\cong  H^{i}(\FF_0,0) =0 \;\;  \mbox{for $i \geqslant 0$},$$whereas
$$H^i(X,\xi+ \varphi^*\Oc_{\FF_0}(0,-3t))\cong H^i( \E\otimes\Oc_{\FF_0}(0,-3t)).$$From \eqref{eq:al-be} and \eqref{eq:al-be3} we have that $\E$ fits in
$$0 \to 2 C_0 + t f  \to {\mathcal E} \to C_0 + t f \to 0.$$Set $\E(0,-3t):= \E\otimes \Oc_{\FF_0}(0,-3t)=\E(0,-3t)$ and tensor the previous exact sequence with $\Oc_{\FF_0}(0,-3t)$, so we get 
\begin{equation}\label{eq:aEbTensor(0,-3t)}
0 \to 2 C_0 -2 t f  \to {\mathcal E}(0,-3t) \to C_0 -2t f \to 0.
\end{equation}
From the cohomology sequence associated to \eqref{eq:aEbTensor(0,-3t)} one gets
\begin{equation}\label{eq:numerata}
h^i( X,\xi+\varphi^*\Oc_{\FF_0}(0,-3t) = h^i(\FF_0,\E(0,-3t))=\left\{
    \begin{array}{ll}
0 &\mbox{ if $i=0,2,3$}\\ 
    10t-5&   \mbox{ if  $i=1$} 
    \end{array}
\right.
 \end{equation}Thus, from \eqref{extension2dualL1}, \eqref{extension2dualM2} and \eqref{seq2}, it follows that $h^2(\cF_2\otimes\cF_2^{\vee}) = 0 = h^3(\cF_2\otimes\cF_2^{\vee})$ and
\begin{eqnarray*}\chi(\cF_2\otimes \cF_2^{\vee} )=\chi( L_1 \otimes \cF_2^{\vee})+\chi(M_2 \otimes \cF_2^{\vee})=-10t+7,
\end{eqnarray*} so, from simplicity of $\cF_2$ and the previous vanishings, one has $h^1(\cF_1\otimes \cF_1^{\vee} )=10t-6$. The conclusion is exactly as in the proof of 
Theorem \ref{prop:rk 2 simple Ulrich vctB e=0;I}.

\smallskip

Case $(3)$ goes as case $(2)$. We consider bundles $\cF_3$ arising from non-trivial extensions in ${\rm Ext}^1(L_2,M_1)$ which, from Cases \ref{cases}, is of dimension $10t-5\geqslant 5$. From 
\cite[Lemma\,4.2]{c-h-g-s} $\cF_3$ is simple so $h^0(\cF_3\otimes \cF_3^{\vee} )=1$. If we tensor by $\cF_3^{\vee}$ the exact sequence defining $\cF_3$ we get  
\begin{eqnarray*}
0 \to M_1 \otimes \cF_3^{\vee}   \to \cF_3\otimes \cF_3^{\vee} \to  L_2 \otimes \cF_3^{\vee} \to 0; 
\end{eqnarray*}taking also the dual sequence of that defining $\cF_3$ and tensoring it with, respectively, $M_1$ and $L_2$ gives
\begin{eqnarray*}
0 \to L_2^{\vee}\otimes M_1(=\xi+ \varphi^*\Oc_{\FF_0}(0,-3t)) \to M_1 \otimes \cF_3^{\vee} \to {\cO}_{X} \to 0
\end{eqnarray*} 
\begin{eqnarray*}
0 \to  {\cO}_{X} \to L_2 \otimes \cF_3^{\vee} \to L_2 \otimes M_1^{\vee}(=-\xi + \varphi^*\Oc_{\FF_0}(0,3t)) \to 0.
\end{eqnarray*} Observing that the above exact sequences are identical to \eqref{extension2dualL1} and \eqref{extension2dualM2}, one computes the remaining cohomologies $h^i(\cF_1\otimes \cF_1^{\vee} )$, $1 \leqslant i \leqslant 3$, exactly as in case $(2)$ and concludes as in the statement.

\smallskip

Similarly,  case $(4)$ goes as case $(1)$. We cosider bundles $\cF_4$ arising from non-trivial extensions in ${\rm Ext}^1(M_2,L_2)$, which is $1$-dimensional. From 
\cite[Lemma\,4.2]{c-h-g-s} it follows that  $\cF_4$ is simple so  $h^0(\cF_4\otimes \cF_4^{\vee} )=1$. To compute the remaining cohomologies $h^i(\cF_4\otimes \cF_4^{\vee} )$, $1 \leqslant i \leqslant 3$,  we tensor with $\cF_4^{\vee}$ the exact sequence defining $\cF_4$ and get
\begin{eqnarray*}
0 \to L_2\otimes \cF_4^{\vee}   \to \cF_4\otimes \cF_4^{\vee} \to  M_2 \otimes \cF_1^{\vee} \to 0;
\end{eqnarray*} taking moreover the dual sequence of that defining $\cF_4$ and tensoring it with, respectively, $L_2$ and $M_2$ gives
\begin{eqnarray*}
0 \to L_2 \otimes M_2^{\vee}(=\xi+ \varphi^*\Oc_{\FF_0}(-3,-t)) \to L_2 \otimes \cF_4^{\vee} \to {\cO}_{X} \to 0
\end{eqnarray*} 
\begin{eqnarray*}
0 \to  {\cO}_{X} \to M_2 \otimes \cF_4^{\vee} \to M_2 \otimes L_2^{\vee}(=-\xi + \varphi^*\Oc_{\FF_0}(3,t)) \to 0. 
\end{eqnarray*}  Observing that the above exact sequences are identical to \eqref{extension1dualL1} and \eqref{extension1bdualM2}, one can conclude exactly as in case $(1)$. 
\end{proof}

\vspace{2mm}

%%%%%%%%%%%%%%%%%%%%%%%%%%%%%%%%%%
%
%% HIGHER RANK Ulrich vector  bundles on $3$-fold scrolls over $\FF_0$
%
%%%%%%%%%%%%%%%%%%%%%%%%%%%%%%%%%%%

\section{Higher-rank sporadic Ulrich bundles on $3$-fold scrolls over $\FF_0$} 
\label{Ulrich higher rk  vb}

In this section we will construct higher-rank, slope-stable, Ulrich vector bundles using both {\em iterative extensions}, by means of the two {\em sporadic} Ulrich line bundles 
\begin{eqnarray}\label{eq:Mi}
M_1 =2\xi+\varphi^*\Oc_{\FF_0}(-1,-t-1) \;\; \mbox{and its Ulrich dual} \;\;  M_2 =\varphi^*\Oc_{\FF_0}(2,3t-1),
\end{eqnarray}as in  {\bf Theorem A}-(a-ii), and {\em deformations} of such vector-bundle extensions in suitable modular families, generalizing the strategy used in \S\,\ref{S:sporadic}  to construct {\em sporadic} rank-$2$ Ulrich bundles. 

Recall that, from \S\,\ref{S:sporadic},  we have 
\begin{eqnarray}\label{eq:dimextLi}
\dim ({\rm Ext}^1 (M_2, M_1)) &=&  h^1(M_1 - M_2) = 6t-3 \geqslant 3.
 %\dim ({\rm Ext}^1 (M_1, M_2)) &= & h^1(M_2 - M_1) = 2t+1 \geqslant 3. \nonumber
 \end{eqnarray}In order to perform {\em recursive constructions}, to ease notation we set once and for all $\cG_1 := M_1$. From \eqref{eq:dimextLi}, the general element of ${\rm Ext}^1 (M_2, \cG_1) = {\rm Ext}^1 (M_2, M_1)$ is a non-splitting extension 
\begin{equation}\label{eq:1r1}
0 \to \cG_1=M_1 \to \cG_2 \to M_2 \to 0,
\end{equation}where $\cG_2:= \cF$, as in the proof of Theorem \ref{prop:rk 2 simple Ulrich vctB e=0;I}, is a rank-$2$ simple vector bundle on $X$, which is Ulrich w.r.t. $\xi$ and with  
\begin{eqnarray*}
c_1(\cG_2) = c_1(M_1) + c_1(M_2) = 2 \xi + \varphi^*\Oc_{\FF_0}(1,2t-2), \; {\rm and}\\ c_2(\cG_2)= c_1(M_1) \cdot c_1(M_2) = \xi \cdot \varphi^*\Oc_{\FF_0}(4,6t-2)- (5t+1)F,
\end{eqnarray*}as in \eqref{eq:chernM}. With a small abuse of notation, we will identify 
extension \eqref{eq:1r1} with the corresponding rank-$2$ vector bundle $\cG_2$, therefore we will state that $[\cG_2] \in {\rm Ext}^1 (M_2, \cG_1)$ is a {\em general element} of this extension space.

If, in the next step,  we considered further extensions ${\rm Ext}^1 (M_2, \cG_2)$,  it is easy to see that the dimension of such an extension space drops by one with respect to that of ${\rm Ext}^1 (M_2, \cG_1)$. Therefore, proceeding in this way,  after finitely many steps we would have ${\rm Ext}^1 (M_2, \cG_r) = \{0\}$, i.e. $\cG_{r+1} = M_2 \oplus \cG_r$, for any $r \geqslant  r_0$, for some positive integer $r_0$. 
To avoid this fact, similarly as in \cite[\S\;4]{cfk1}, we proceed by taking {\em alternating sporadic extensions}, namely 
 \[0 \to \cG_2 \to \cG_3 \to M_1 \to 0,\;\; 0 \to \cG_3 \to \cG_4 \to M_2 \to 0,\; \ldots , 
  \]
  and so on, that is, defining
  \begin{equation} \label{eq:jr}
    \epsilon_r: =
    \begin{cases}
      1, & \mbox{if $r$ is odd}, \\
      2, & \mbox{if $r$ is even},
    \end{cases}
  \end{equation}
  we take successive $[\cG_{r}] \in \Ext^1(M_{\epsilon_{r}},\cG_{r-1})$, for all $r \geqslant2$, defined by:
  \begin{equation}\label{eq:1}
0 \to \cG_{r-1} \to \cG_{r} \to M_{\epsilon_{r}} \to 0.
 \end{equation}

The fact that we always get  {\em non--trivial} such extensions, for any $r \geqslant 2$, will be proved in Corollary \ref{cor:Corollario al Lemma 4.2} below. In any case all vector bundles $\cG_{r}$, recursively defined as in \eqref{eq:1}, are of rank $r$ and Ulrich w.r.t. $\xi$, since extensions of Ulrich bundles w.r.t. $\xi$ are again Ulrich w.r.t. $\xi$. From the fact that any $\cG_r$ is recursively defined, Chern classes of $\cG_r$, for any $r \geqslant 2$, are obtained as linear combination, with coefficients depending on $r$, of $c_1(M_i)$ or $c_1(M_i) \cdot c_1(M_j)$, for $1 \leqslant i,j \leqslant 2$. Precisely, Chern classes of $\cG_r$  are: 
\begin{equation} \label{eq:c1rcaso0}
    c_1(\cG_r): =
    \begin{cases} 
      (r +1)\xi + \varphi^*\Oc_{\FF_0}(0, -2t) + \varphi^*\Oc_{\FF_0}\left(\frac{(r-3)}{2},r(t-1) \right), & \mbox{if $r$ is odd}, \\ 
      r \xi + \varphi^*\Oc_{\FF_0}\left(\frac{r}{2}, r(t-1)\right), & \mbox{if $r$ is even},  
    \end{cases}
  \end{equation}
     \begin{eqnarray*}
c_2(\cG_r) =
    \begin{cases} 
       \xi \cdot  \varphi^*\Oc_{\FF_0}\left(2r^2-2, (2t-1)r^2-2t+1 \right) -\frac{(r-1)(2rt+r+14t-3)}{2} F, & \mbox{if $r\geqslant 3$ is odd}, \\
       \xi \cdot \varphi^*\Oc_{\FF_0}\left(2r^2-2r, r(2rt-r-t+1) \right)-\frac{r(2rt+r+t-1)}{2} F, & \mbox{if $r$ is even},  
    \end{cases}
      \end{eqnarray*}
     \begin{eqnarray*}
c_3(\cG_r) =
    \begin{cases} 
       4r^3t-2r^3-8r^2t+4r^2-4rt+2r+8t-4,   & \mbox{if $r\geqslant 3$ is odd}, \\
       4r^3t-2r^3-10r^2t+6r^2+4rt-4r, & \mbox{if $r\geqslant 4$ is even}. 
    \end{cases}
    \end{eqnarray*}For any $r \geqslant 1$, from  \eqref{slope}, the slope of $\cG_r$ w.r.t. $\xi$ is $\mu(\cG_r) = 13t - 3$. Moreover, from Theorem \ref{thm:stab}-(a), any such bundle $\cG_r$ is strictly semistable and slope-semistable, being obtained by extensions of Ulrich bundles.

\begin{rem}\label{rem:differentmodsp} {\normalfont We want to stress {\em non-sporadic/non-mixed} extensions studied in \cite{fa-fl2}, namely extensions using line bundle pair $(L_1, L_2)$ as in {\bf Theorem A}-(a-i), existing for any pair $(b,k)$ satisfying \eqref{(iii)} and \eqref{eq:rem:assAB} when $e=0$, have different Chern classes $c_2$ and $c_3$. Thus, even when we restrict 
	to the cases $(b,k) = (2t, 3t)$, with $t \geqslant 1$ any integer, moduli spaces 
	determined by (deformations of) bundles $\cG_r$ as in \eqref{eq:1} are different moduli spaces from those constructed in \cite{fa-fl2} when $(b,k) = (2t, 3t)$, $t \geqslant 1$. }
\end{rem}

%\bigskip	
%{\color{red} NOTA: la $c_1(G_r)$ con estensioni a partire da $M_1$ e $M_2$ coincide con la $c_1$ di altre estensioni solo nel caso in cui si hanno estensioni con $L_1$ e $L_2$, $r$  pari  e   $b_0=2t, k_0=3t$ (visto che $M_1$ e $M_2$  si hanno solo per tali valori). 
%Ad ogni modo in questo caso le $c_2$ sono diverse. Ho calcolato la  $c_2(G_r)$ con estensioni a partire da $L_1$ e $L_2$  nel caso $r$ pari e viene  diversa da quella in \eqref{eq:c1rcaso0}, sempre caso pari.
%\medskip
%Infatti in quest'ultimo caso
% \begin{eqnarray*}
%c_2(\cU_r) =
%          \xi \cdot \varphi^*\Oc_{\FF_0}\left(2r^2-2r, r(2rt-r-2t+1) \right)-\frac{r(4rt-r-t+1)}{2} F & \mbox{if $r$ is even}. 
%      \end{eqnarray*}
%}

Among other things, the next result will allow us to prove the aforementioned claim that,  from \eqref{eq:1}, we always get non-trivial extensions (cf.\,Corollary \ref{cor:Corollario al Lemma 4.2}). This fact, together with what proved in Lemma \ref{lem:new} below, will also imply the existence of simple, 
so indecomposable, Ulrich vector bundles w.r.t. $\xi$ for any rank $r \geqslant 2$ (cf. Corollary \ref{cor:new} below).

\begin{lemma} \label{lemma:1} Let  $M$ denote any of the two line bundles 
$M_1$ and $M_2$ as in \eqref{eq:Mi}. Then, for all integers $r \geqslant1$, we have
    \begin{itemize}
    \item[(i)] $h^2(\cG_r \otimes M^{\vee})= h^3(\cG_r \otimes M^{\vee}) = 0$,
      \item[(ii)] $h^2(\cG_r^{\vee} \otimes M)=h^3(\cG_r^{\vee} \otimes M) = 0 $,
      \item[(iii)] $h^1(\cG_r \otimes M_{\epsilon_{r+1}}^{\vee})\geqslant{\min}\{6t-3,\;2t+1\} \geqslant3$.
      \end{itemize}
  \end{lemma}
  
	\begin{proof} For $r=1$, by definition,  we have $\cG_1 = M_1$; therefore $\cG_1 \otimes M^{\vee}$ and 
	$\cG_1^{\vee} \otimes M$  are either equal to $\mathcal O_{X}$, if $M=M_1$, or equal to $M_1 - M_2$ and $M_2 - M_1$, respectively, if $M=M_2$. Therefore (i) and (ii) hold true by computations as in \S\,\ref{S:sporadic}. As for (iii), by \eqref{eq:jr} we have that $M_{\epsilon_{2}} = M_2$ thus 
	$h^1(\cG_1 \otimes M_{2}^{\vee}) = h^1(M_1 - M_2) = 6t - 3$, as is \S\;\ref{S:sporadic}, the latter being always greater than or equal to ${\min}\{6t-3,\;2t+1\} \geqslant 3$. 
	
	Therefore, we will assume $r \geqslant2$ and proceed by induction. 
	
	Regarding (i), since it holds for $r=1$, assuming it holds for $r-1$, then 
	by tensoring \eqref{eq:1} with $M^{\vee}$ we get that
	\begin{eqnarray*}
     h^j(\cG_{r} \otimes M^{\vee}) =0, \;\; j=2,3,
    \end{eqnarray*}
	because $h^j(\cG_{r-1} \otimes M^{\vee}) = 0$, for $j=2,3$, by inductive hypothesis whereas 
		$h^j(M_{\epsilon_{r}} \otimes M^{\vee}) = 0 $, for $j=2,3$, since $M_{\epsilon_{r}} \otimes M^{\vee}$ 
		 is either $\mathcal O_{X}$, or $M_2-M_1$, or $M_1-M_2$. 
		 
		 A similar reasoning, tensoring the dual of \eqref{eq:1} by $M$, proves (ii).

    To prove (iii), tensor \eqref{eq:1} by $M_{\epsilon_{r+1}}^{\vee}$ and use that $h^2(\cG_{r-1}\otimes M_{\epsilon_{r+1}}^{\vee})=0$ by (i). Thus 
		we have the surjection
$$H^1(\cG_r \otimes M_{\epsilon_{r+1}}^{\vee}) \twoheadrightarrow H^1(M_{\epsilon_{r}} \otimes M_{\epsilon_{r+1}}^{\vee}),$$which implies that 
$h^1(\cG_r \otimes M_{\epsilon_{r+1}}^{\vee}) \geqslant h^1(M_{\epsilon_{r}} \otimes M_{\epsilon_{r+1}}^{\vee})$. According to the 
parity of $r$, we have that $M_{\epsilon_{r}} \otimes M_{\epsilon_{r+1}}^{\vee}$ equals 
either $M_1 - M_2$ or $M_2 - M_1$. From computations as in \S\;\ref{Ulrich rk 2 vb}, 
$h^1(M_1-M_2) = 6t-3$ whereas $h^1(M_2-M_1) = 2t+1$. Notice that 
\[
{\rm min} \{6t-3,\;2t+1\} = \begin{cases}
      6t-3 = 2t+1= 3, & \mbox{if $t=1$}, \\
			2t+1 \geqslant5, & \mbox{if $t\geqslant2$}.
    \end{cases}
		\]Therefore one concludes.
		\end{proof}

  \begin{cor}\label{cor:Corollario al Lemma 4.2}  For any integer $r \geqslant 1$ there exist on $X$ rank-$r$ vector bundles $\cG_r$, which are Ulrich w.r.t. $\xi$, with Chern classes as in \eqref{eq:c1rcaso0}, of slope  w.r.t. $\xi$ given by $\mu(\cG_r) = 13t - 3$ and which arise as non-trivial extensions as in \eqref{eq:1} if $r \geqslant 2$. 
	\end{cor}
	\begin{proof} For $r=1$, we have $\cG_1 = M_1$ and the statement holds true from {\bf Theorem A}-(a-ii) and computations in \S\;\ref{Ulrich rk 2 vb}. 
	
	 For any $r \geqslant2$, notice that $\Ext^1(M_{\epsilon_{r}}, \cG_{r-1}) \cong H^1(\cG_{r-1} \otimes M_{\epsilon_{r}}^{\vee})$. Therefore, from Lemma \ref{lemma:1}-(iii) there exist non--trivial  extensions as in \eqref{eq:1}, which therefore give rise to bundles $\cG_r$ which are  Ulrich with respect to $\xi$,  whose Chern classes are  exactly as those computed in \eqref{eq:c1rcaso0}. Since they are Ulrich bundles, the statement about their slope w.r.t. $\xi$ directly follows from \eqref{slope}.    
	\end{proof}
 
 From Corollary \ref{cor:Corollario al Lemma 4.2}, at any step we can always pick {\em non--trivial} extensions of the form \eqref{eq:1} and 
we will henceforth do so. Next result uses similar strategies as in \cite[Lemma\,4.3]{fa-fl2}.

\begin{lemma} \label{lemma:2}  Let $r \geqslant 1$  be an integer. Then we have 
    \begin{itemize}
    \item[(i)] $h^1(\cG_{r+1} \otimes M_{\epsilon_{r+1}}^{\vee})=h^1(\cG_r \otimes M_{\epsilon_{r+1}}^{\vee})-1$,
    \item[(ii)] $h^1(\cG_r \otimes M_{\epsilon_{r+1}}^{\vee})= 
		\begin{cases}
      \frac{(r+1)}{2} h^1(M_1-M_2) - \frac{(r-1)}{2} = \frac{r+1}{2}(6t-3) - \frac{(r-1)}{2}, & \mbox{if $r$ is odd}, \\
			\frac{r}{2} h^1(M_2-M_1) - \frac{(r-2)}{2} = \frac{r}{2}(2t+1) - \frac{(r-2)}{2}, & \mbox{if $r$ is even}.
    \end{cases}$
		\item[(iii)] $h^2(\cG_r \otimes \cG_r^{\vee}) = h^3(\cG_r \otimes \cG_r^{\vee})=0$,
		
\item[(iv)] $\chi(\cG_r \otimes M_{\epsilon_{r+1}}^{\vee})= 
\begin{cases}
      \frac{(r+1)}{2} (1- h^1(M_1-M_2)) -1  =  \frac{(r+1)}{2} (4 - 6t) -1, & \mbox{if $r$ is odd}, \\
			\frac{r}{2} (1- h^1(M_2-M_1)) = - r t , & \mbox{if $r$ is even}.
    \end{cases}$

\item[(v)] $\chi(M_{\epsilon_{r}} \otimes \cG_r^{\vee})= 
\begin{cases}
      \frac{(r-1)}{2} (1- h^1(M_1-M_2)) + 1  = \frac{(r-1)}{2} (4 - 6 t) + 1, & \mbox{if $r$ is odd}, \\
			\frac{r}{2} (1- h^1(M_2-M_1)) = - r t, & \mbox{if $r$ is even}.
    \end{cases}$
\item[(vi)]  $\chi(\cG_r \otimes \cG_r^{\vee})= 
\begin{cases}
  \scriptstyle    \frac{(r^2 -1)}{4} (2-h^1(M_1-M_2)-h^1(M_2-M_1)) + 1  = \frac{(r^2 -1)}{4}(4 - 8t) + 1, & \mbox{if $r$ is odd}, \\
		\scriptstyle	\frac{r^2}{4} (2- h^1(M_1-M_2)-h^1(M_2-M_1)) = \frac{r^2}{4} (4- 8t) , & \mbox{if $r$ is even}.
    \end{cases}$
\end{itemize}
\end{lemma}

\begin{proof}  (i) Consider the exact sequence \eqref{eq:1}, with $r$ replaced by $r+1$. From $\Ext^1(M_{\epsilon_{r+1}},\cG_r)\cong H^1(\cG_r \otimes M_{\epsilon_{r+1}}^{\vee})$ and from the fact that the exact sequence defining $\cG_{r+1}$ is constructed by taking a non--zero vector $[\cG_{r+1}] \in \Ext^1(M_{\epsilon_{r+1}},\cG_r)$, it follows that the coboundary map 
\begin{eqnarray*}
H^0(\mathcal O_{X}) \stackrel{\partial}{\longrightarrow } H^1(\cG_r \otimes M_{\epsilon_{r+1}}^{\vee})
\end{eqnarray*} arising from 
\begin{eqnarray}\label{eq:dag}
0 \to \cG_r \otimes M_{\epsilon_{r+1}}^{\vee} \to \cG_{r+1}\otimes M_{\epsilon_{r+1}}^{\vee} \to  \mathcal O_{X} \to 0,
 \end{eqnarray} is non--zero, so it is injective; thus, (i) follows from the cohomology of \eqref{eq:dag}.

\medskip

\noindent
(ii) Here one uses induction on $r$. For $r=1$, the right hand side of the formula
yields $6t-3$ which is $h^1(\cG_1 \otimes M_2^{\vee})=h^1(M_1-M_2)$,  see 
 \eqref{eq:dimextLi}. When otherwise  $r=2$, the right hand side of the formula  is $2t+1$ which is $ h^1(\cG_2 \otimes M_1^{\vee}) = h^1(M_2 -M_1) = 2t+1$, as seen in  \eqref{eq:dimextLi}, and from the exact sequence 
\begin{eqnarray*}
0 \to \mathcal O_{X} \to \cG_2 \otimes M_1^{\vee} \to M_2 - M_1 \to 0,
\end{eqnarray*}
obtained by \eqref{eq:1} with $r=2$ and tensored with $M_1^{\vee}$, and the fact that 
$h^j(\mathcal O_{X}) = 0$, for $j=1,2$.

Assuming by induction that formula as in (ii) holds true up to some given integer $r \geqslant 2$, one has to show that it holds also for $r+1$. Considering \eqref{eq:1}, with $r$ replaced by $r+1$, and tensoring it by $M_{\epsilon_{r+2}}^{\vee}$ we thus 
obtain
\begin{eqnarray} \label{eq:dagdag}
0 \to \cG_r\otimes M_{\epsilon_{r+2}}^{\vee} \to \cG_{r+1}\otimes M_{\epsilon_{r+2}}^{\vee} \to  M_{\epsilon_{r+1}}\otimes M_{\epsilon_{r+2}}^{\vee} \to 0
 \end{eqnarray}

If $r$ is even, then by definition $M_{\epsilon_{r+2}} = M_2$ whereas $M_{\epsilon_{r+1}} =M_1$. Thus $h^0(M_{\epsilon_{r+1}}\otimes M_{\epsilon_{r+2}}^{\vee})= h^0(M_1 - M_2) = 0$ and  
$h^1(M_{\epsilon_{r+1}}\otimes M_{\epsilon_{r+2}}^{\vee})=h^1(M_1 - M_2) = 6t -3$.  On the other hand, by Lemma \ref{lemma:1}-(i), 
$h^2(\cG_r\otimes M_{\epsilon_{r+2}}^{\vee})=0$. Thus, from \eqref{eq:dagdag}, we get: 
\begin{eqnarray*}
 h^1(\cG_{r+1} \otimes M_{\epsilon_{r+2}}^{\vee})=6t-3 +h^1(\cG_r\otimes M_{\epsilon_{r+2}}^{\vee})=6t-3+
h^1(\cG_r\otimes M_{\epsilon_{r}}^{\vee}), 
 \end{eqnarray*}
 as $r$ and $r+2$ have the same parity. Using (i),  we have $h^1(\cG_r\otimes M_{\epsilon_{r}}^{\vee}) = h^1(\cG_{r-1}\otimes M_{\epsilon_{r}}^{\vee}) -1$ therefore, 
by inductive hypothesis with $r-1$ odd, we have $h^1(\cG_{r-1}\otimes M_{\epsilon_{r}}^{\vee}) = \frac{r}{2} (6t-3) - \frac{(r-2)}{2}$. 
 Summing up, we have 
\begin{eqnarray*}
h^1(\cG_{r+1} \otimes M_{\epsilon_{r+2}}^{\vee}) = (6t-3) + \frac{r}{2} (6t-3) - \frac{(r-2)}{2} -1,
\end{eqnarray*}
which is easily seen to be equal to the right hand side expression in (ii), when $r$ is replaced by $r+1$. 

If  $r$ is odd, the same holds for $r+2$ whereas $r+1$ is even. In this case  $M_{\epsilon_{r+2}}= M_1$, 
$M_{\epsilon_{r+1}} =M_2$ so $h^1(M_{\epsilon_{r+1}}\otimes M_{\epsilon_{r+2}}^{\vee})=h^1(M_2 - M_1) = 2t+1$ and one applies the same procedure 
as in the previous case. 

\medskip

\noindent
(iii)  We again use induction on $r$. For $r=1$, formula (iii) states that $h^j(M_1 -M_1)=h^j(\mathcal O_{X})=0$, for $j=2,3$,  which is  certainly true. 

Assume now that (iii) holds  up to some integer $r \geqslant 1$; we have to prove that it holds also for $r+1$. 
Consider the exact sequence \eqref{eq:1}, where $r$ is replaced by $r+1$, and tensor it by $\cG_{r+1}^{\vee}$. From this we get that, for $j=2,3$, 
\begin{eqnarray} \label{eq:zaniolo}
 h^j( \cG_{r+1} \otimes \cG_{r+1}^{\vee}) \leqslant h^j( \cG_{r} \otimes \cG_{r+1}^{\vee}) + h^j(M_{\epsilon_{r+1}} \otimes \cG_{r+1}^{\vee}) = h^j( \cG_{r} \otimes \cG_{r+1}^{\vee}),
\end{eqnarray} the latter equality follows from $h^j( M_{\epsilon_{r+1}} \otimes \cG_{r+1}^{\vee}) =0$, $j =2,3$, as in 
Lemma \ref{lemma:1}-(ii).  

Consider the dual exact sequence of \eqref{eq:1}, where $r$ is replaced by $r+1$, and tensor it by $\cG_{r}$. Thus, 
Lemma \ref{lemma:1}-(i) yields that, for $j=2,3$, one has 
\begin{eqnarray} \label{eq:pogba}
h^j( \cG_{r} \otimes \cG_{r+1}^{\vee}) \leqslant h^j(\cG_r \otimes M_{\epsilon_{r+1}}^{\vee})+h^j(\cG_{r} \otimes \cG_{r}^{\vee})= h^j(\cG_{r} \otimes \cG_{r}^{\vee}).
\end{eqnarray} Now \eqref{eq:zaniolo}--\eqref{eq:pogba} and the inductive hypothesis yield
$h^j( \cG_{r+1} \otimes \cG_{r+1}^{\vee})=0$, for $j=2,3$, as desired.

\medskip

\noindent
(iv)  For $r=1$,  (iv) reads $\chi(M_1-M_2)= - h^1(M_1-M_2) = 3 - 6t$, which is true since $h^j(M_1-M_2)=0$ for $j=0,2,3$. For $r=2$, (iv) reads $\chi(\cG_2 \otimes M_1^{\vee}) = 1 - h^1(M_2-M_1) =  - 2t$ and this holds true because 
 if we take the exact sequence \eqref{eq:1}, with $r=2$, tensored by $M_1^{\vee}$ then 
\begin{eqnarray*} 
\chi(\cG_2 \otimes M_1^{\vee}) = \chi (\mathcal O_{X}) + \chi (M_2-M_1) = 1 - h^1(M_2-M_1) = 1 - (2t+1),
\end{eqnarray*} 
as $h^j(M_2-M_1)= 0$ for $j=0,2,3$. 

Assume now that the formula holds up to a certain integer $r \geqslant 2$, we have to prove that it also holds for $r+1$.  
From  \eqref{eq:dagdag} we get 
\begin{eqnarray*} 
\chi(\cG_{r+1}\otimes M_{\epsilon_{r+2}}^{\vee}) = \chi(\cG_r\otimes M_{\epsilon_{r+2}}^{\vee}) + \chi(M_{\epsilon_{r+1}}\otimes M_{\epsilon_{r+2}}^{\vee}).
\end{eqnarray*}

If $r$ is even, the same is   true for $r+2$ whereas $r+1$ is odd. Therefore, 
\begin{equation} \label{eq:maz1}
  \chi(\cG_{r+1}\otimes M_{\epsilon_{r+2}}^{\vee}) = \chi(\cG_r\otimes M_{2}^{\vee})+ \chi(M_{1} - M_{2}) = 
	\chi(\cG_r\otimes M_{2}^{\vee})-h^1(M_1-M_2).
  \end{equation} Then  \eqref{eq:dag}, with $r$ replaced by $r-1$, yields
  \begin{equation}\label{eq:maz2}
    \chi(\cG_r\otimes M_{2}^{\vee})=\chi(\cG_{r-1}\otimes M_{2}^{\vee})+\chi(\mathcal O_{X})= \chi(\cG_{r-1}\otimes M_{2}^{\vee})+1.
  \end{equation}
Substituting \eqref{eq:maz2} into \eqref{eq:maz1} and using the inductive hypothesis with $r-1$ odd, we get
  \begin{eqnarray*}
    \chi(\cG_{r+1}\otimes M_{2}^{\vee}) &=& \chi(\cG_{r-1}\otimes M_{2}^{\vee})+1-h^1(M_1-M_2) \\
    & = & \frac{r}{2} (1- h^1(M_1-M_2))-h^1(M_1-M_2) \\
		 & = & \frac{(r+2)}{2} (1- h^1(M_2-M_1)) -1,   
\end{eqnarray*} proving that the formula holds also for $r+1$ odd.

Similar procedure can be used to treat the case when $r$ is odd. In this case, $M_{\epsilon_{r+1}} = M_2$ whereas $M_{\epsilon_{r+2}} = M_1$. Thus, 
from the above computations, 
\begin{eqnarray*}
\chi(\cG_{r+1}\otimes M_1^{\vee}) = \chi(\cG_r\otimes M_{1}^{\vee})+ \chi(M_2 - M_1) = \chi(\cG_r\otimes M_{1}^{\vee})- h^1(M_2 - M_1).
\end{eqnarray*}
As in the previous case, 
$\chi(\cG_r\otimes M_{1}^{\vee}) = 1 + \chi (\cG_{r-1} \otimes M_{1}^{\vee})$ so, applying inductive hypothesis with $r-1$ even, we get 
$\chi(\cG_r\otimes M_{1}^{\vee}) = 1 + \frac{(r-1)}{2} (1 - h^1(M_2-M_1))$. Adding up all these quantities, we get 
\begin{eqnarray*}\chi(\cG_{r+1}\otimes M_{\epsilon_{r+2}}^{\vee})= \chi(\cG_{r+1}\otimes M_1^{\vee}) = \frac{r+1}{2} (1 - h^1(M_2 -M_1)),
\end{eqnarray*}
so formula (iv) 
holds true also for $r+1$ even.

\medskip

\noindent
(v) For $r=1$, (v) reads $\chi(M_1 -M_1) = \chi(\mathcal O_{X}) =1$, which is as stated. 
For $r=2$, (v) reads  $\chi(M_2\otimes \cG_2^{\vee})= 1 - h^1(M_2-M_1)$, which is once again as stated, as it follows from the dual of sequence \eqref{eq:1} tensored by $M_2$.

Assume now that the formula holds up to a certain integer $r \geqslant 2 $ and we need to proving it  for $r+1$.  Dualizing \eqref{eq:1}, replacing  $r$  by  $r+1$ and tensoring it  by $M_{\epsilon_{r+1}}$
 we find that
\begin{eqnarray} \label{eq:maz3}
  \chi(M_{\epsilon_{r+1}} \otimes \cG_{r+1}^{\vee} ) & = &   \chi(M_{\epsilon_{r+1}} \otimes M_{\epsilon_{r+1}}^{\vee})+ \chi(M_{\epsilon_{r+1}}\otimes \cG_{r}^{\vee}) \\
  \nonumber & = &  \chi(\mathcal O_{X_n})+\chi(M_{\epsilon_{r+1}}\otimes \cG_{r}^{\vee}) 
 =  1+\chi(M_{\epsilon_{r+1}}\otimes \cG_{r}^{\vee}).
  \end{eqnarray}
   The  dual of sequence \eqref{eq:1}, with $r$ replaced by $r-1$, tensored by $M_{\epsilon_{r+1}}$ yields
  \begin{equation}
    \label{eq:maz4}
\chi(M_{\epsilon_{r+1}}\otimes \cG_{r}^{\vee})=\chi(M_{\epsilon_{r+1}} \otimes M_{\epsilon_{r}}^{\vee})+ \chi(M_{\epsilon_{r+1}}\otimes \cG_{r-1}^{\vee}).
\end{equation}
   Substituting \eqref{eq:maz4} into \eqref{eq:maz3} and using the fact that $r+1$ and $r-1$ have the same parity, we get 
\begin{eqnarray*}	\chi(M_{\epsilon_{r+1}} \otimes \cG_{r+1}^{\vee} ) = 1 + \chi(M_{\epsilon_{r+1}} \otimes M_{\epsilon_{r}}^{\vee}) + 
\chi(M_{\epsilon_{r-1}}\otimes \cG_{r-1}^{\vee}).\end{eqnarray*}

If $r$ is even,  then $\chi(M_{\epsilon_{r+1}} \otimes M_{\epsilon_{r}}^{\vee}) = \chi(M_1-M_2) = - h^1(M_1-M_2)$ 
whereas, from the inductive hypothesis with $r-1$ odd, $\chi(M_{\epsilon_{r-1}}\otimes \cG_{r-1}^{\vee}) = 1 + \frac{(r-2)}{2} (1 - h^1(M_1-M_2))$. Thus
\begin{eqnarray*}\chi(M_{\epsilon_{r+1}} \otimes \cG_{r+1}^{\vee} ) = 1 -  h^1(M_1-M_2) + 1 + \frac{(r-2)}{2} (1 - h^1(M_1-M_2)),
\end{eqnarray*}
the latter 
equals $1 +  \frac{r}{2} (1 - h^1(M_1-M_2))$, proving that the formula holds also for $r+1$ odd.

If $r$ is odd, the strategy is similar; in this case one has 
$\chi(M_{\epsilon_{r+1}} \otimes M_{\epsilon_{r}}^{\vee}) = \chi(M_2-M_1) = - h^1(M_2-M_1)$ and, by the inductive hypothesis with $r-1$ even,
$\chi(M_{\epsilon_{r-1}}\otimes \cG_{r-1}^{\vee}) = \frac{(r-1)}{2} (1 - h^1(M_2-M_1))$ so one can conclude. 

\medskip

\noindent
(vi) We first check the given formula for $r=1,2$. We have $\chi(\cG_1 \otimes \cG_1^{\vee})=\chi(M_1-M_1)=\chi(\mathcal O_{X})=1$, which fits with the given formula for $r=1$. 
From \eqref{eq:1}, with $r=2$, tensored by $\cG_2^{\vee}$ we get
\begin{equation}\label{eq:cr1}
  \chi(\cG_2 \otimes \cG_2^{\vee})=\chi(M_1 \otimes \cG_2^{\vee})+\chi(M_2 \otimes \cG_2^{\vee})\stackrel{(v)}{=} \chi(M_1 \otimes \cG_2^{\vee}) + 1 - h^1(M_2-M_1).
\end{equation} 
From the dual of  \eqref{eq:1}, with $r=2$,  tensored by $M_1$ we get
\begin{equation} \label{eq:cr2}
  \chi(M_1 \otimes \cG_2^{\vee})=\chi(M_1 -M_1)+\chi(M_1 -M_2)=\chi(\mathcal O_{X})- h^1(M_1-M_2) =1-h^1(M_1-M_2).
\end{equation} Combining \eqref{eq:cr1} and \eqref{eq:cr2}, we get 
\begin{eqnarray*}
 \chi(\cG_2 \otimes \cG_2^{\vee})= 2 -h^1(M_1-M_2) - h^1(M_2-M_1),
 \end{eqnarray*} 
 which again fits with the given formula for $r=2$.

Assume now that the given formula is valid up to a certain integer $r \geqslant2$; we need to prove it holds for $r+1$. 
From \eqref{eq:1}, in which $r$ is replaced by $r+1$, tensored by $\cG_{r+1}^{\vee}$ and successively the dual of  \eqref{eq:1}, with $r$ replaced by $r+1$,  
tensored by $\cG_r$ we get
\begin{eqnarray*}
\chi(\cG_{r+1} \otimes \cG_{r+1}^{\vee}) & = & \chi(\cG_{r} \otimes \cG_{r}^{\vee}) +\chi(\cG_{r} \otimes M_{\epsilon_{r+1}}^{\vee})+ \chi(M_{\epsilon_{r+1}} \otimes \cG_{r+1}^{\vee}).
\end{eqnarray*}

If $r$ is even, then $r+1$ is odd and $M_{\epsilon_{r+1}} = M_1$. From (v) with $(r+1)$ odd, we get 
$\chi(M_{\epsilon_{r+1}} \otimes \cG_{r+1}^{\vee}) = 1 + \frac{r}{2} (1 - h^1(M_1-M_2))$, whereas from (iv) with $r$ even
$\chi(\cG_{r} \otimes M_{\epsilon_{r+1}}^{\vee}) = \frac{r}{2} (1 - h^1(M_2-M_1))$. Finally, by the inductive hypothesis with $r$ even, 
$\chi(\cG_{r} \otimes \cG_{r}^{\vee}) = \frac{r^2}{4} (2 - h^1(M_1-M_2) - h^1(M_2-M_1))$. Summing--up the three quantities, one gets 
\begin{eqnarray*}
\chi(\cG_{r+1} \otimes \cG_{r+1}^{\vee}) = 1 + \frac{(r+1)^2-1}{4} (2 - h^1(M_1-M_2) - h^1(M_2-M_1)),
\end{eqnarray*}
proving that the formula holds for $r+1$ odd.
 
If  $r$ is odd, then 
$\chi(M_{\epsilon_{r+1}} \otimes \cG_{r+1}^{\vee}) = \frac{r+1}{2} (1 - h^1(M_2-M_1))$, as it follows from (v) with $(r+1)$ even, whereas
$\chi(\cG_{r} \otimes M_{\epsilon_{r+1}}^{\vee}) = \frac{(r+1)}{2} (1 - h^1(M_1-M_2)) - 1$, as predicted by (iv) with $r$ odd. Finally, form the inductive hypothesis with 
$r$ odd, we have 
$\chi(\cG_{r} \otimes \cG_{r}^{\vee}) = 1 + \frac{(r^2-1)}{4} (2 - h^1(M_1-M_2)- h^1(M_2-M_1))$. If we add up the three quantities, we get 
\begin{eqnarray*}
\chi(\cG_{r+1} \otimes \cG_{r+1}^{\vee}) = \frac{(r+1)^2}{4} (2 - h^1(M_1-M_2) - h^1(M_2-M_1)),
\end{eqnarray*}
finishing the proof. 
\end{proof}

Notice some fundamental remarks arising from the first step of the previous iterative contruction in \eqref{eq:1}, which turns out from the proof of Theorem \ref{prop:rk 2 simple Ulrich vctB e=0;I}. We set $\cG_1 = M_1$, which is an Ulrich line bundle w.r.t. $\xi$, of slope $\mu = \mu(M_1) = 13 t -3$; by considering 
non--trivial extensions \eqref{eq:1r1}, $\cG_2$ turned out to be a simple (so indecomposable) bundle, as it follows from \cite[Lemma\,4.2]{c-h-g-s} and from 
the fact that $\cG_1 = M_1$ and $M_{\epsilon_2} = M_2$ are both slope--stable, of the same slope $\mu = 13 t -3$ w.r.t. $\xi$ and non-isomorphic line bundles. 
By construction, $\cG_2$ turned out to be morover Ulrich, so strictly semistable, of slope $\mu = 13 t -3$. On the other hand, 
in the proof of Theorem \ref{prop:rk 2 simple Ulrich vctB e=0;I} we showed that $\cG_2$ deforms, in an irreducible modular family, to a slope-stable Ulrich bundle $\cU_2 := \cU$, of the same slope w.r.t. $\xi$ given by $\mu = 13t-3$, same Chern classes 
$c_i(\cU_2) = c_1(\cG_2)$, $1 \leqslant i \leqslant 2$. By semi-continuity in the irreducible modular family, cohomological properties as in Lemma \ref{lemma:1}-(i-ii) and Lemma \ref{lemma:2}-(iii-iv-v-vi) hold true when $\cG_2$ therein is replaced by $\cU_2$. Therefore, from $h^2(\cU_2 \otimes \cU_2^{\vee}) = 0$ and simplicity of $\cU_2$, by Proposition \ref{casanellas-hartshorne} (cf.\,also\,\cite[Prop.\,2.10]{c-h-g-s}) and dimensional computation of $h^1(\cU_2 \otimes \cU_2^{\vee})$, $\cU_2$ is a general point of the corresponding modular family and it is also a smooth point, so that the irreducible modular family is generically smooth. Up to shrinking to the open set of smooth points of such an irreducible modular family, we may consider a smooth modular family of simple, slope-stable, Ulrich bundles and the GIT-quotient relation restricted to such a smooth modular family gives rise to an \'etale cover of an open dense subset of the modular component $\mathcal M = \mathcal M(2)$ (by the very definition of modular family, cf. \cite[pp.\,1250083-9/10]{c-h-g-s}), which is therefore generically smooth of the same dimension of the modular family, i.e.  $h^1(\cU_2 \otimes \cU_2^{\vee})$, and whose general point is $[\cU_2]$, described in Theorem \ref{prop:rk 2 simple Ulrich vctB e=0;I} (cf. also {\bf Theorem B-sporadic cases} in Introduction).

By induction we can therefore assume that, up to a given integer $r \geqslant 3$, we have already constructed a generically smooth, irreducible modular component 
$\mathcal M(r-1)$ of the moduli space of bundles of rank $(r-1)$, which are Ulrich w.r.t. $\xi$, with Chern classes $c_i := c_i(\cG_{r-1})$ as in \eqref{eq:c1rcaso0} (where in the formulas therein $r$ is obviously replaced by $r-1$), for $1 \leqslant i \leqslant 3$,  and whose general point $[\cU_{r-1}] \in \mathcal M(r-1)$ is slope-stable, of slope w.r.t. $\xi$ given by 
$\mu(\cU_{r-1}) = 13t-3$ and that satiesfies Lemma \ref{lemma:1}-(i-ii) and Lemma \ref{lemma:2}-(iii-iv-v-vi). Consider therefore extensions 
\begin{equation} \label{eq:estensionM}
  0 \to \cU_{r-1} \to \cF_r \to M_{\epsilon_r} \to 0, 
\end{equation}
with $[\cU_{r-1}] \in \mathcal M(r-1)$ general and with $M_{\epsilon_r}$ defined as in \eqref{eq:jr}, \eqref{eq:1}, according to the parity of $r$. 
Notice that 
\begin{eqnarray*}
{\rm Ext}^1(M_{\epsilon_r}, \cU_{r-1}) \cong H^1(\cU_{r-1} \otimes M_{\epsilon_r}^{\vee}).
\end{eqnarray*}

\begin{lemma}\label{lem:new} In the above set-up, one has 
$$h^1(\cU_{r-1} \otimes M_{\epsilon_r}^{\vee}) \geqslant {\rm min} \{6t-4,\; 2t\} \geqslant 2.$$In particular, 
${\rm Ext}^1(M_{\epsilon_r}, \cU_{r-1})$ contains non-trivial extensions as in \eqref{eq:estensionM}.
\end{lemma}

\begin{proof} By assumption $\cU_{r-1}$ satisfies Lemma \ref{lemma:1}-(i), so 
one has 
\begin{equation}\label{eq:porca1}
h^j(\cU_{r-1} \otimes M_{\epsilon_r}^{\vee}) = h^j(\cG_{r-1} \otimes M_{\epsilon_r}^{\vee}) = 0, \; j = 2,3.
\end{equation} Similarly, as by assumptions it also satisfies Lemma \ref{lemma:2}-(iv) 
(with $r$ replaced by $r-1$), one has 
\begin{equation}\label{eq:(*)}
\chi(\cU_{r-1} \otimes M_{\epsilon_r}^{\vee}) = \chi(\cG_{r-1} \otimes M_{\epsilon_r}^{\vee}).
\end{equation} Thus, equality in \eqref{eq:(*)}, together with \eqref{eq:porca1}, reads
\begin{eqnarray*}
h^0(\cU_{r-1} \otimes M_{\epsilon_r}^{\vee}) - h^1(\cU_{r-1} \otimes M_{\epsilon_r}^{\vee}) = h^0(\cG_{r-1} \otimes M_{\epsilon_r}^{\vee} ) - h^1( \cG_{r-1} \otimes M_{\epsilon_r}^{\vee}),
\end{eqnarray*}
namely
\begin{equation}\label{eq:porca3}
h^1(\cU_{r-1} \otimes M_{\epsilon_r}^{\vee}) = h^1( \cG_{r-1} \otimes M_{\epsilon_r}^{\vee}) - \left(h^0(\cG_{r-1} \otimes M_{\epsilon_r}^{\vee} ) - h^0(\cU_{r-1} \otimes M_{\epsilon_r}^{\vee} ) \right),
\end{equation}
where $h^1(\cG_{r-1} \otimes M_{\epsilon_r}^{\vee}) \geqslant {\rm min} \{6t-3,\;2t +1\} \geqslant 3$, as from Lemma \ref{lemma:1}-(iii) where $r$ is replaced by $r-1$. 
We claim that the following equality

\begin{equation}\label{eq:porcavacca}
h^0(\cG_{r-1} \otimes M_{\epsilon_r}^{\vee}) = 
\left\{
\begin{array}{ccl}
0 & {\rm if} &  r\; {\rm \; even} \\
1 & {\rm if} & r \; {\rm odd} 
\end{array}
\right.
\end{equation}holds true. 

Assume for a moment that \eqref{eq:porcavacca} has been proved;  since $\cU_{r-1}$ is slope-stable, 
of the same slope as $M_{\epsilon_r}$, and $\cU_{r-1}$ is not isomorphic to $M_{\epsilon_r}$, then 
$h^0(\cU_{r-1} \otimes M_{\epsilon_r}^{\vee} ) = 0$ as any non-zero homomorphism $M_{\epsilon_r} \to \cU_{r-1}$ 
should be an isomorphism. Thus, using \eqref{eq:porca3}, for any $r \geqslant 2$ 
one gets therefore 
\begin{eqnarray*}
h^1(\cU_{r-1} \otimes L_{\epsilon_r}^{\vee}) \geqslant h^1(\cG_{r-1} \otimes L_{\epsilon_r}^{\vee} ) -1
\end{eqnarray*}  
which, together with Lemma \ref{lemma:1}-(iii), proves the statement. 

Thus, we are left with the proof of \eqref{eq:porcavacca}. To prove it, we will use induction on $r$. 

If $r=2$, then $\cG_1 = M_1$, $M_{\epsilon_2} = M_2$, thus 
$h^0(\cG_1 \otimes M_2^{\vee}) = h^0(M_1-M_2) = 0$, as it follows from \eqref{extension1dualM1} and from \eqref{eq:calcoliutili2}. If otherwise $r=3$, then $\cG_{r-1} = \cG_2$ 
as in \eqref{eq:1r1} whereas $M_{\epsilon_3} = M_1$, as in \eqref{eq:jr}. Thus, tensoring \eqref{eq:1r1} by $M_1^{\vee}$, one gets
\begin{eqnarray*}
0 \to \Oc_{X} \to \cG_2 \otimes M_1^{\vee} \to M_2 - M_1 \to 0;
\end{eqnarray*}
since 
$h^0(M_2 - M_1) = 0$, from \eqref{extension1dualM2} and from \eqref{eq:calcoliutili1}, then $h^0( \cG_2 \otimes M_1^{\vee}) = h^0(\Oc_{X}) = 1$.

Assume therefore that, up to some integer $r-2 \geqslant 2$,  \eqref{eq:porcavacca} holds true and take $\cG_{r-1}$ a non-trivial extension as in \eqref{eq:1}, with $r$ 
replaced by $r-1$, namely 
\begin{equation}\label{eq:porca*}
0 \to \cG_{r-2} \to \cG_{r-1} \to M_{\epsilon_{r-1}} \to 0. 
\end{equation}

If $r$ is even, then $r-2$ is even and $r-1$ is odd, in particular $M_{\epsilon_{r-1}} = M_1$ and $M_{\epsilon_{r}} = M_2$. Thus, tensoring \eqref{eq:porca*} with $M_{\epsilon_{r}}^{\vee} = 
M_2^{\vee}$ gives 
\begin{eqnarray*}
0 \to \cG_{r-2} \otimes M_2^{\vee} \to \cG_{r-1} \otimes M_2^{\vee} \to M_1 - M_2 \to 0.
\end{eqnarray*}
Since $h^0(M_1 - M_2) = 0$ then 
\begin{eqnarray*}
h^0(\cG_{r-1} \otimes M_{\epsilon_{r}}^{\vee})  = h^0(\cG_{r-1} \otimes M_2^{\vee}) = h^0(\cG_{r-2} \otimes M_2^{\vee}).\end{eqnarray*}
On the other hand, 
by \eqref{eq:1}, with $r$ replaced by $r-2$, namely 
\begin{equation}\label{eq:porca**}
0 \to \cG_{r-3} \to \cG_{r-2} \to M_{\epsilon_{r-2}} \to 0,
\end{equation}
we have $M_{\epsilon_{r-2}} = M_2$, since $r-2$ is even as $r$ is. Thus, tensoring \eqref{eq:porca**} with $M_{\epsilon_r}^{\vee}$ and taking into account that $r$ is even, one gets 
\begin{eqnarray*}
0 \to \cG_{r-3} \otimes M_2^{\vee} \to \cG_{r-2}  \otimes M_2^{\vee}\to \Oc_{X} \to 0.
\end{eqnarray*}
Notice that 
$\cG_{r-3} \otimes M_2^{\vee} = \cG_{r-3} \otimes M_{\epsilon_{r-2}}^{\vee}$ thus, since $r-3$ is odd, $h^0(\cG_{r-3} \otimes M_2^{\vee}) = 0$ by induction and by \eqref{eq:porcavacca}. On the other hand, 
the coboundary map   
\begin{eqnarray*}
H^0(\Oc_{X}) \cong {\mathbb C} \stackrel{\partial}{\longrightarrow} H^1(\cG_{r-3} \otimes  M_2^{\vee})= H^1(\cG_{r-3} \otimes  M_{\epsilon_{r-2}}^{\vee}) \cong {\rm Ext}^1(M_{\epsilon_{r-2}}, \cG_{r-3})
\end{eqnarray*}
is non-zero since, by iterative construction, $\cG_{r-2}$ is taken to be a non-trivial extension; therefore $\partial$ is injective which implies 
$h^0(\cG_{r-2} \otimes  M_2^{\vee}) = 0$ and so $h^0(\cG_{r-1} \otimes  M_{\epsilon_r}^{\vee}) = 0$, as desired.

Assume now $r$ to be odd thus, $M_{\epsilon_{r}} = M_1$ whereas $M_{\epsilon_{r-1}} = M_2$. Tensoring \eqref{eq:porca*} with $M_1^{\vee}$ gives 
\begin{eqnarray*}
0 \to \cG_{r-2} \otimes M_1^{\vee} \to \cG_{r-1} \otimes M_1^{\vee} \to M_2 - M_1 \to 0.
\end{eqnarray*}
As $h^0(M_2 - M_1) = 0$, then 
\begin{eqnarray*}
h^0(\cG_{r-1} \otimes M_{\epsilon_{r}}^{\vee})  = h^0(\cG_{r-1} \otimes M_1^{\vee}) = h^0(\cG_{r-2} \otimes M_1^{\vee}).
\end{eqnarray*}
Since $r$ is odd, then also $r-2$ is odd and one gets
\begin{eqnarray*}
0 \to \cG_{r-3} \otimes M_1^{\vee} \to \cG_{r-2}  \otimes M_1^{\vee}\to \Oc_{X} \to 0.
\end{eqnarray*}
Notice that 
$h^0(\cG_{r-3} \otimes M_1^{\vee} ) = h^0(\cG_{r-3} \otimes M_{\epsilon_{r-2}}^{\vee}) = 1$, as it follows from \eqref{eq:porcavacca} with $r$ replaced by $r-2$ which is odd since $r$ is. 
On the other hand, the fact that $\cG_{r-2}$ arises from a non--trivial extension implies as before that the coboundary map 
\begin{eqnarray*}
H^0(\Oc_{X}) \cong {\mathbb C} \stackrel{\partial}{\longrightarrow} H^1(\cG_{r-3} \otimes  M_1^{\vee})= H^1(\cG_{r-3} \otimes  M_{\epsilon_{r-2}}^{\vee}) \cong {\rm Ext}^1(M_{\epsilon_{r-2}}, \cG_{r-3})
\end{eqnarray*}
is once again injective. This gives  $h^0(\cG_{r-2} \otimes  M_1^{\vee}) = h^0(\cG_{r-3} \otimes  M_1^{\vee}) = 1$, which implies 
 $h^0(\cG_{r-1} \otimes  M_{\epsilon_r}^{\vee}) = 1$. This concludes the proof of the Lemma.  \end{proof}

Lemma \ref{lem:new} ensures that there exist non-trivial extensions arising from \eqref{eq:estensionM}. Then one has the following consequence.

\begin{cor}\label{cor:new} For a given $r \geqslant 2$, assume that $\mathcal M(r-1) \neq \emptyset$ and  that $[\cU_{r-1}] \in \mathcal M(r-1)$ general corresponds to a rank-$r$ vector bundle, which is Ulrich w.r.t. $\xi$ and slope-stable, of slope $\mu(\cU_{r-1}) = 13t-3$ (where, for $r=2$, $\cU_1 = \cG_1 = M_1$ and 
$\mathcal M(1) = \{M_1\}$ is a singleton). Consider $M_{\epsilon_r}$ as in \eqref{eq:jr} and \eqref{eq:1}. 

Then $[\cF_r] \in {\rm Ext}^1(M_{\epsilon_{r}}, \cU_{r-1})$ general is a rank-$r$ vector bundle, which is simple, so indecomposable, Ulrich w.r.t. $\xi$, with Chern classes as in \eqref{eq:c1rcaso0}, of slope w.r.t. $\xi$ given by $\mu(\cF_r) = 13t-3$ and with $h^j(\cF_r \otimes \cF^{\vee}_r) = 0$, for $2 \leqslant j \leqslant 3$. 
\end{cor}

\begin{proof} Since $\cU_{r-1}$ and $M_{\epsilon_r}$ are both Ulrich w.r.t. $\xi$, then 
it immediately follows that $\cF_r$ is of rank $r$, Ulrich w.r.t. $\xi$ and of slope 
as stated, by \eqref{slope}. 

Now, since $\cU_{r-1}$ is slope-stable, with $\cU_{r-1}$ not isomorphic to $M_{\epsilon_r}$ (if $r>2$, ${\rm rk}(\cU_{r-1}) > 1 = {\rm rk}(M_{\epsilon_r})$, if otherwise $r=2$, $\cU_{1} = M_1$ and $M_{\epsilon_2} = M_2$ are not isomorphic), and since moreover $\cU_{r-1}$ and $M_{\epsilon_r}$ have the same slope
$\mu = 13t -3$ then, by \cite[Lemma\;4.2]{c-h-g-s}, general $[\mathcal F_r] \in {\rm Ext}^1(M_{\epsilon_{r}}, \cU_{r-1})$ corresponds to a simple, so indecomposable, rank-$r$ vector bundle, with Chern classes as in \eqref{eq:c1rcaso0}.

To prove the assertions on cohomology groups, consider the dual sequence of \eqref{eq:estensionM} and tensor it by $\cF_r$, which gives 
$$0 \to M^{\vee}_{\epsilon_r}  \otimes \cF_r \to \cF^{\vee}_r \otimes \cF_r \to \cU^{\vee}_{r-1}  \otimes \cF_r \to 0.$$Thus, 
\begin{equation}\label{eq:cohomFr}
h^j (\cF^{\vee}_r \otimes \cF_r) \leqslant h^j(M^{\vee}_{\epsilon_r}  \otimes \cF_r ) + h^j(\cU^{\vee}_{r-1}  \otimes \cF_r ), \;2 \leqslant j \leqslant 3. 
\end{equation} On the other hand taking \eqref{eq:estensionM} tensored, respectively, by $M^{\vee}_{\epsilon_r}$ and $\cU^{\vee}_{r-1}$ gives 
$$0 \to M^{\vee}_{\epsilon_r}  \otimes \cU_{r-1} \to M^{\vee}_{\epsilon_r} \otimes \cF_r \to \mathcal O_X \to 0\;\; {\rm and}\;\; 0 \to \cU^{\vee}_{r-1} \otimes \cU_{r-1} \to \cU^{\vee}_{r-1} \otimes \cF_r \to \cU^{\vee}_{r-1} \otimes M_{\epsilon_r} \to 0,$$from which one has 
$$h^j(M^{\vee}_{\epsilon_r}  \otimes \cF_r ) \leqslant h^j(M^{\vee}_{\epsilon_r}  \otimes \cU_{r-1} ) + h^j(\mathcal O_X), \; 2 \leqslant j \leqslant 3,$$and 
$$h^j(\cU^{\vee}_{r-1} \otimes \cF_r ) \leqslant h^j(\cU^{\vee}_{r-1} \otimes \cU_{r-1} ) + h^j(\cU^{\vee}_{r-1} \otimes M_{\epsilon_r}), \; 2 \leqslant j \leqslant 3.$$Thus from $h^j(\mathcal O_X) =0$, for $j =2,3$, from Lemmas \ref{lemma:1}-(i-ii) and 
\ref{lemma:2}-(iii) and inductive assumptions on $\cU_{r-1}$,  one deduces that 
$h^j(M^{\vee}_{\epsilon_r}  \otimes \cF_r ) = h^j(\cU^{\vee}_{r-1} \otimes \cF_r ) = 0$, 
for $j=2,3$, which plugged in \eqref{eq:cohomFr} gives $h^j(\cF^{\vee}_r \otimes \cF_r)= 0$, for $2 \leqslant j \leqslant 3$, as stated. 
\end{proof}

Take therefore $[\cF_r] \in {\rm Ext}^1(M_{\epsilon_{r}}, \cU_{r-1})$ general. From 
Corollary \ref{cor:new} we know that $\cF_r$ is simple with $h^2(\cF^{\vee}_r \otimes \cF_r)= 0$. Therefore, by \cite[Proposition\;10.2]{c-h-g-s}, $\mathcal F_r$ admits a smooth modular family which, with a small abuse of notation, we denote by $\mathcal M(r)$ as the modular component of {\bf Theorem B-sporadic cases} in Introduction. Indeed, by definition of smooth modular family as in \cite[pp.\,1250083-9/10]{c-h-g-s}, an open dense subset of it will be  an \'etale cover of the modular component $\mathcal M(r)$ we are going to contruct; for this reason and to avoid heavy notation, they will be sometimes identified. 

For $r \geqslant 2$, such a smooth modular family $\mathcal M(r)$ contains a subscheme, denoted by $\mathcal M(r)^{\rm ext}$, which parametrizes bundles $\cF_r$ arising from non--trivial extensions as in \eqref{eq:estensionM}.

\begin{lemma} \label{lemma:genUr}  Let $r \geqslant2$  be an integer and let 
$\cU_r$ be a general member of the modular family $\mathcal M(r)$. Then $\cU_r$ is a vector bundle of rank $r$, which is Ulrich with respect to $\xi$, with slope w.r.t. $\xi$ given by $\mu:= \mu(\cU_r) = 13t - 3$, and with  Chern classes as \eqref{eq:c1rcaso0}. 
   
   Moreover $\cU_r$ is simple, in particular indecomposable, with 
  \begin{itemize}
\item[(i)]  $\chi(\cU_r \otimes \cU_r^{\vee})=   \begin{cases} \frac{(r^2 -1)}{4}(4 - 8t) + 1, &  \mbox{if $r$ is odd,}
      \\
		 \frac{r^2}{4} (4-8t) , & \mbox{if $r$ is even}.
    \end{cases}$
 \item[(ii)] $h^j(\cU_r \otimes \cU_r^{\vee})=0$, for $j =2,3$.  
  \end{itemize}
\end{lemma}
  
\begin{proof} Since $\mathcal F_r$ is of rank $r$ and Ulrich w.r.t. $\xi$, the same holds true for the general member $[\cU_r] \in \mathcal M(r)$, since 
Ulrichness is an open property in irreducible families as $\mathcal M(r)$. In particular, from \eqref{slope}, one has $\mu(\cU_r) = \mu(\mathcal F_r) = \mu(\cU_{r-1})$. For the same reasons, Chern classes of $\cU_r$ coincide with those of $\cF_r$ which, in turn, are as in \eqref{eq:c1rcaso0}.

Since $\mathcal F_r$ is simple, as proved in Corollary \ref{cor:new}, by semi-continuity on $\mathcal M(r)$ one has also $h^0( \cU_r \otimes \cU_r^{\vee}) = 1$,  i.e. $\cU_r$ is simple, in particular it is indecomposable.

Property $(ii)$ follows by semi-continuity in the smooth modular family $\mathcal M(r)$ when $\cU_r$ specializes to $\cF_r$ and from Corollary \ref{cor:new}.

Property $(i)$ follows from Lemma \ref{lemma:2}-$(vi)$, since the given $\chi$ depends only on the Chern classes of $X$, which are fixed, and on the Chern classes of the two factors,  which in turn are those as $\cF_r$ and so of $\cG_r$ as well. 
\end{proof}

We will prove that the general member $\cU_r$ in the smooth modular family $\mathcal M(r)$ is a slope--stable bundle w.r.t. $\xi$. To prove it, we will make use of  the following auxiliary result, whose proof is identical to that of \cite[Lemma\,4.6]{fa-fl2}, to which the iterested reader is referred.

\begin{lemma} \label{lemma:uniquedest} Let $r \geqslant2$ be an integer and assume that 
the element  $\cF_r$ of the subscheme $\mathcal M(r)^{\rm ext}$ sits in a non--splitting sequence like  \eqref{eq:estensionM}, with $[\cU_{r-1}] \in \mathcal M(r-1)$ general. Then, if $\cD$ is a destabilizing subsheaf of $\cF_r$, then $\cD^{\vee} \cong  \cU_{r-1}^{\vee}$ and $(\cF_r/\cD)^{\vee} \cong M_{\epsilon_r}^{\vee}$; if furthermore $\cF_r/\cD$ is torsion--free, then  $\cD \cong  \cU_{r-1}$ and $\cF_r/\cD \cong M_{\epsilon_r}$. 
\end{lemma}

\begin{proof} See the proof of \cite[Lemma\,4.6]{fa-fl2}.        
\end{proof}

\begin{lemma} \label{lemma:dimU} Let $r \geqslant2$  be an integer. Take $[\cU_{r-1}] \in \mathcal M(r-1)$ a general point, where $\mathcal M(r-1)$ is the modular component as in {\bf Theorem B-sporadic case}. Then, the modular family $\mathcal M(r)$ as in Lemma \ref{lemma:genUr} is  generically  smooth, of dimension 
	\begin{eqnarray*}
	\dim (\mathcal M(r) ) = \begin{cases} \frac{(r^2 -1)}{4}(8t -4), & \mbox{if $r$ is odd}, \\
			 \frac{r^2}{4} (8t-4) +1 , & \mbox{if $r$ is even}.
    \end{cases}
    \end{eqnarray*}
    Furthermore $\mathcal M(r)$ properly contains the locally closed subscheme $\mathcal M(r)^{\rm ext}$, namely \linebreak 
		$\dim(\mathcal M(r)^{\rm ext}) < \dim(\mathcal M(r))$.  
\end{lemma}

\begin{proof} Let $\mathcal U_r$ be a general member of the smooth modular family 
$\mathcal M(r)$. From Lemma \ref{lemma:genUr}, one has  $h^0(\cU_r \otimes \cU_r^{\vee})=1$, i.e. it is simple, and $h^j(\cU_r \otimes \cU_r^{\vee})=0$ for $j=2,3$.

From the fact that $h^2(\cU_r \otimes \cU_r^{\vee})=0$, it follows that  the modular family $\mathcal M(r)$ is  generically smooth of dimension $\dim (\mathcal M(r)) = h^1(\cU_r \otimes \cU_r^{\vee})$ (cf. 
e.g. \cite[Prop.\;2.10]{c-h-g-s}). On the other hand, since $h^3(\cU_r \otimes \cU_r^{\vee})=0$ and $h^0(\cU_r \otimes \cU_r^{\vee})=1$, we have  
$h^1(\cU_r \otimes \cU_r^{\vee})  =  -\chi(\cU_r \otimes \cU_r^{\vee})+1$. Therefore, the formula concerning $\dim (\mathcal M(r))$ directly follows from Lemma \ref{lemma:genUr}-(i), since the given $\chi$ depends only on the Chern classes of $X$, which are fixed, and on the Chern classes of the two factors,  which in turn are those as $\cF_r$ and so of $\cG_r$ as well. 

Similarly $[\cU_{r-1}] \in \mathcal M(r-1)$ general is by assumptions slope-stable, so in particular simple, thus it satisfies $h^0(\cU_{r-1} \otimes \cU_{r-1}^{\vee})=1$. Thus, using Lemma \ref{lemma:genUr}-(ii), the same reasoning as above 
	shows that
  \begin{equation}\label{eq:dimUr-1}
\dim (\mathcal M(r-1))=  h^1(\cU_{r-1} \otimes \cU_{r-1}^{\vee}) = - \chi(\cU_{r-1} \otimes \cU_{r-1}^{\vee}) +1,   
\end{equation}where $\chi (\cU_{r-1} \otimes \cU_{r-1}^{\vee})$ as in Lemma \ref{lemma:genUr}-(i) (with $r$ replaced by $r-1$).   Morover, by \eqref{eq:porca3}, we have 
  \begin{equation}\label{eq:dimextv}
    \dim (\Ext^1(M_{\epsilon_r},\cU_{r-1}))= h^1(\cU_{r-1} \otimes M_{\epsilon_r}^{\vee})\leqslant h^1(\cG_{r-1} \otimes M_{\epsilon_r}^{\vee}),
  \end{equation} where the latter is as in Lemma \ref{lemma:2}-(ii) (with $r$ replaced by $r-1$). 
	Therefore, by the very definition of $\mathcal M(r)^{\rm ext}$ and by \eqref{eq:dimUr-1}-\eqref{eq:dimextv}, we have 
	\begin{eqnarray*}
  \dim (\mathcal M(r)^{\rm ext}) & \leqslant &  \dim (\mathcal M(r-1)) + \dim (\mathbb P (\Ext^1(M_{\epsilon_r},\cU_{r-1})) \\
          & = & - \chi(\cU_{r-1} \otimes \cU_{r-1}^{\vee}) +1 + h^1(\cU_{r-1} \otimes M_{\epsilon_r}^{\vee}) -1 \\
& \leqslant & - \chi(\cU_{r-1} \otimes \cU_{r-1}^{\vee}) + h^1(\cG_{r-1} \otimes M_{\epsilon_r}^{\vee}).
\end{eqnarray*} On the other hand, from the above discussion, 
\begin{eqnarray*}
\dim (\mathcal M(r)) = - \chi(\cU_{r} \otimes \cU_{r}^{\vee}) +1.
\end{eqnarray*}
Therefore to prove that $\dim (\mathcal M(r)^{\rm ext}) < \dim (\mathcal M(r))$ 
it is enough to  show that for any integer $r \geqslant2$ the following inequality 
\begin{eqnarray*} - \chi(\cU_{r-1} \otimes \cU_{r-1}^{\vee}) + h^1(\cG_{r-1} \otimes M_{\epsilon_r}^{\vee}) < - \chi(\cU_{r} \otimes \cU_{r}^{\vee}) +1
\end{eqnarray*}
holds true. 
Notice that the previous inequality reads also  
\begin{equation}\label{eq:ineqpal}
- \chi(\cU_{r} \otimes \cU_{r}^{\vee}) +1 + \chi(\cU_{r-1} \otimes \cU_{r-1}^{\vee}) - h^1(\cG_{r-1} \otimes M_{\epsilon_r}^{\vee}) >0,
\end{equation} which is satisfied for any $r \geqslant2$, as we can easily see. 

Indeed use Lemmas \ref{lemma:genUr}-(i) and \ref{lemma:2}-(ii): if $r$ is even, the left hand side of \eqref{eq:ineqpal} reads $ \frac{(r)}{2}(8t-4) + 2 + \frac{(r-2)}{2}$ which obviously is positive since 
$r \geqslant2$  and  $t \ge 1$; if $r$ is odd, then $r \geqslant3$ and the left hand side of \eqref{eq:ineqpal} reads $\frac{r-1}{2} (6t-5) + \frac{(r-3)}{2}$ which obviously is positive 
under the assumptions $r \geqslant3$, $t \geqslant1$. 
\end{proof}

We can now prove slope--stability w.r.t. $\xi$ of the general member of modular family $\mathcal M(r)$.

\begin{prop}\label{prop:slopstab} Let $r \geqslant 1$ be an integer. The general member $\cU_r$ in the modular family  $\mathcal M(r)$ is a bundle which is slope--stable w.r.t. $\xi$. 
\end{prop}
\begin{proof} We use induction on $r$, the result being obviously true for $r=1$, where $\cU_1 = M_1$, $\mathcal M(1) = \{M_1\}$ is a singleton, 
and $\mathcal M(1)^{\rm ext} = \emptyset$.

Assume therefore $r \geqslant2$ and, by contradiction, that the general member of $\mathcal M(r)$ were not slope--stable, whereas the general point $[\cU_{r-1}] \in \mathcal M(r-1)$ of the modular component corresponds to a bundle which is slope-stable w.r.t. $\xi$. Then, similarly as in 
\cite[Prop.\;4.7]{cfk1}, we may find a one-parameter family of bundles $\{\cU_r^{(t)}\}$ over the unit disc $\Delta$ such that $\cU_r^{(t)}$ is a general member of $\mathcal M(r)$ for $t \neq 0$ and $\cU_r^{(0)}$ lies in $\mathcal M(r)^{\rm ext}$, and such that we have a destabilizing sequence
\begin{equation} \label{eq:destat1} 
    0 \to \cD^{(t)} \to \cU_r^{(t)} \to \cQ^{(t)} \to 0 
  \end{equation}
  for $t \neq 0$, which we can take to be saturated, that is, such that $\cQ^{(t)}$ is torsion free, whence so that $\cD^{(t)}$ and $\cQ^{(t)}$ are (Ulrich) vector bundles  (see \cite[Thm. 2.9]{c-h-g-s} or \cite[(3.2)]{b}).

  The limit of $\mathbb P(\cQ^{(t)}) \subset \mathbb P(\cU_r^{(t)})$ defines a subvariety of $\mathbb P(\cU_r^{(0)})$ of the same dimension as  $\mathbb P(\cQ^{(t)})$,   whence a coherent sheaf 
	$\cQ^{(0)}$ of rank ${\rm rk}(\cQ^{(t)})$ with a surjection $\cU_r^{(0)} \to \cQ^{(0)}$. Denoting by $\cD^{(0)}$ its kernel, we have
${\rm rk}(\cD^{(0)})={\rm rk}(\cD^{(t)})$ and $c_1(\cD^{(0)})=c_1(\cD^{(t)})$. Hence, \eqref{eq:destat1} specializes to a destabilizing sequence for $t=0$. 
Lemma \ref{lemma:uniquedest} yields  that ${\cD^{(0)}}^{\vee}$ (respectively,
${\cQ^{(0)}}^{\vee}$) is the dual of a member of $\mathcal{M}(r-1)$ (resp., the dual of $M_{\epsilon_r}$).
It follows that ${\cD^{(t)}}^{\vee}$ (resp., ${\cQ^{(t)}}^{\vee}$)
is a deformation of the dual of a member of  $\mathcal M(r-1)$  (resp., a deformation of $M_{\epsilon_r}^{\vee}$), whence that $\cD^{(t)}$ is a deformation of a member of  $\mathcal{M}(r-1)$, as both are locally free, and $\cQ^{(t)} \cong M_{\epsilon_r}$, for the same reason. 

In other words, the general member of $\mathcal M(r)$ 
is an extension of $M_{\epsilon_r}$ by a member of
 $\mathcal{M}(r-1)$. Hence $\mathcal M(r)=\mathcal M(r)^{\rm ext}$, contradicting Lemma \ref{lemma:dimU}. 
\end{proof}

The collection of the previous results gives the following

\begin{theo}\label{thm:general0}  Let  $(X, \xi) \cong \scrollcal{E}$ be a $3$-fold  scroll over $\FF_0$, with $\mathcal E = \mathcal E_0$ satisfying Assumptions \ref{ass:AB}  Let $\varphi: X \to \FF_0$ be the scroll map and $F$ be the $\varphi$-fiber. Let $r \geqslant2$ be any integer. 
Then, for any integer $t \geqslant 1$, the moduli space of rank-$r$ vector bundles $\cU_r$ on $X$ which are Ulrich w.r.t. $\xi$ and with Chern classes as in \eqref{eq:c1rcaso0}, 
    is not empty and it contains a generically smooth component $\mathcal M(r)$ of dimension 
		\begin{eqnarray*}
		\dim (\mathcal M(r) ) = \begin{cases} \frac{(r^2 -1)}{4}(8t -4), & \mbox{if $r$ is odd}, \\
			 \frac{r^2}{4} (8t-4) +1 , & \mbox{if $r$ is even}.
    \end{cases}
    \end{eqnarray*} Moreover the general point $[\cU_r] \in \mathcal M(r)$  
corresponds to a  slope-stable vector bundle, of slope w.r.t. $\xi$ given by 
$\mu(\cU_r) = 13t -3$.

In particular, there are no slope-stable-Ulrich-rank gaps on $X_0$ w.r.t. the chosen Chern classes.
\end{theo}
\begin{proof} It directly follows from Theorem  \ref{prop:rk 2 simple Ulrich vctB e=0;I}, \eqref{eq:c1rcaso0}, \eqref{slope} and from 
Lemmas \ref{lemma:genUr}, \ref{lemma:dimU} and Proposition \ref{prop:slopstab} where, as already mentioned, by a small abuse of notation we have used the same symbol $\mathcal M(r)$  for the smooth modular family which, via GIT-quotient, gives rise by an \'etale cover to an open, dense subset of the generically smooth, irreducible component 
of the corresponding moduli space. 
\end{proof}

%%%%%%%%%%%%%%%%%%%%%%%%%%%%%%%%%%
%
%% HIGHER RANK mixed Ulrich vector  bundles on $3$-fold scrolls over $\FF_0$
%
%%%%%%%%%%%%%%%%%%%%%%%%%%%%%%%%%%%

\section{Higher-rank mixed Ulrich bundles on $3$-fold scrolls over $\FF_0$} 
\label{Ulrichhighermixed} 

In this section we briefly mention how to construct other positive-dimensional {\em sporadic modular components} of moduli spaces which arise by a similar approach as in \S\,\ref{Ulrich higher rk  vb}, but with the use of {\em mixed pairs} $(L_i, M_j)$, $1 \leqslant i,j \leqslant 2$, as in {\bf Theorem A}. Such components will be therefore different from those determined in {\bf Theorem B-sporadic cases}. 

Looking back to rank-$2$ cases, our starting point will be to consider the positive-dimensional component $\mathcal M:= \mathcal M(2)$, where $\mathcal M(2)$ as either in Theorem \ref{prop:rk 2 simple Ulrich vctB e=0;I}, or in Theorem \ref{prop:rk 2 simple Ulrich vctB e=0;II}-(2) and (3). 

If $[\mathcal U_2] \in \mathcal M(2)$ is a general point, then one can consider as in \eqref{eq:estensionM} extensions ${\rm Ext}^1(A, \mathcal U_2)$, where $A$ runs among all possible choices $A = L_1, L_2, M_1, M_2$ as in {\bf Theorem A}-(a). This gives rise to rank-$3$ vector bundles, which are Ulrich w.r.t. $\xi$ on $X$  and whose Chern classes are given by $$c_1:= c_1(\mathcal U_2) + c_1(A),\;\; c_2: = c_2(\mathcal U_2) + c_1(\mathcal U_2)\cdot c_1(A),\;\; c_3:= c_2(\mathcal U_2)\cdot c_1(A).$$Whenever one gets non-trivial extensions, one can compute cohomological properties of the general bundle in such an extension space similarly as done in Lemmas \ref{lemma:1}, \ref{lemma:2} and \ref{lem:new}. In such a case, reasoning as in Corollary \ref{cor:new}, such a general bundle turns out to be simple and whose cohomological properties computed implying that it sits in a smooth modular family as in Proposition \ref{casanellas-hartshorne}. 

Then one deduces properties of the general bundle $\mathcal U_3$ in the given modular family as done in Lemma \ref{lemma:genUr}. Using same strategies as in Lemma \ref{lemma:dimU} and in Proposition \ref{prop:slopstab}, such $\mathcal U_3$ gives rise to a general point of a generically smooth modular component $\mathcal M(3)$ of computed dimension. Then, one can recursively apply the same procedure to this new $\mathcal U_3$ by pairing it in extensions via an Ulrich line bundle $A$, where $A$ runs among all possible choices $A = L_1, L_2, M_1, M_2$ as in {\bf Theorem A}-(a).

It is clear from the above description that the number of possible cases to study at each step grows as the rank increases. For this reason, since the procedures to use in any  of the cases  are exactly as in the previous section, we will limit ourselves to considering one significant case giving rise to {\em sporadic modular components} which are different from those determined in {\bf Theorem B-sporadic cases}; these will be called 
{\em extra sporadic modular components}.

\bigskip

Let us therefore consider our starting step as given by $[\mathcal U_2] \in \mathcal M(2)$ general as in Theorem \ref{prop:rk 2 simple Ulrich vctB e=0;I}. In particular one has $$c_1(\cU_2)= c_1(M_1) + c_1(M_2) \;\; {\rm and} \;\; c_2(\cU_2)= c_1(M_1)\cdot c_1(M_2),$$namely $\mathcal U_2$ arises as a suitable deformation of vector bundle extensions by means of the {\em sporadic pair} $(M_1, M_2)$. We then need to take an Ulrich line bundle $A$, running among all possible choices $A = L_1, L_2, M_1, M_2$ as in {\bf Theorem A}-(a).

From what proved in \S\;\ref{Ulrich higher rk  vb}, we can avoid to extend $\cU_2$ by $A=M_1$, since this case has been already considered therein.

\smallskip

\noindent
$(i)$ if $A = L_1$, one computes that $\dim( {\rm Ext}^1(L_1, \mathcal U_2)) = 
h^1(\cU_2 (-L_1)) =0$, as it follows by semi-continuity and by \eqref{extension1} tensored by $- L_1$. Therefore, extension via $L_1$  does not provide an indecomposable bundle and we therefore get rid of this case.

\smallskip

\noindent
$(ii)$ if $A = L_2$, to compute $\dim( {\rm Ext}^1(L_2, \mathcal U_2)) = 
h^1(\cU_2 (-L_2))$ we first observe that,  since $\cU_2$ and $L_2$ are both slope-stable and of the same slope, then $h^0(\cU_2 (-L_2)) = 0$. We consider \eqref{extension1} tensored by $-L_2$ which gives 
$$0 \to M_1-L_2 = \xi +\varphi^*\Oc_{\FF_0}(0,-3t) \to \cF (-L_2) \to M_2-L_2 = - \xi+\varphi^*\Oc_{\FF_0}(1,t) \to 0,$$from which one gets 
$h^j(\cF (-L_2)) = 0$, for $2 \leqslant j \leqslant 3$, therefore the same holds true for $h^j(\cU_2 (-L_2)) = 0$, $2 \leqslant j \leqslant 3$, by semi-continuity. Thus, by the invariance of Euler characteristic in irreducible flat families, one gets $h^1(\cU_2 (-L_2)) = h^1(\cF (-L_2))$ where the latter can be computed by the previous exact sequence. Since $h^1(M_2-L_2) = h^1( - \xi+\varphi^*\Oc_{\FF_0}(1,t) ) = 0$, then 
$h^1(\cF (-L_2)) = h^1(M_1-L_2)$ which, by Leray, equals $h^1(\mathbb F_0, \mathcal E\otimes \mathcal O_{\mathbb F_0}(-3t\,f)) 
= 10t-5$, as in \eqref{eq:numerata}.

Therefore, $\dim( {\rm Ext}^1(L_2, \mathcal U_2)) = 10t-5$ so we have non-trivial extensions of rank $3$ and, as in Corollary \ref{cor:new},  $[\cF_3] \in {\rm Ext}^1(L_2, \mathcal U_2)$ general turns out to be simple, so indecomposable, Ulrich and with $$c_1(\cF_3) = c_1(\cU_2) + c_1(L_2) = c_1(M_1) + c_1(M_2) + c_1(L_2) = 3\xi + \varphi^*\Oc_{\FF_0}(3,2t-3)$$ (similar computations for the other Chern classes).

\smallskip

\noindent
$(iii)$ if otherwise $A = M_2$, as above one computes that $\dim( {\rm Ext}^1(M_2, \mathcal U_2)) = h^1(\cU_2 (-M_2)) = 6t-4$, and $[\cF'_3] \in {\rm Ext}^1(M_2, \mathcal U_2)$ general turns out to be simple, so indecomposable, Ulrich and with $$c_1(\cF'_3) = c_1(\cU_2) + c_1(M_2) 
= c_1(M_1) + 2 c_1(M_2)= 2\xi + \varphi^*\Oc_{\FF_0}(4,6t-2)$$ (similar computations for the other Chern classes).

\smallskip 

To go on, in any of the above cases,  we should now consider pairings of the general rank-$3$ with an Ulrich line bundle $A$, with $A$ running once again among all possible choices $A = L_1, L_2, M_1, M_2$ as in {\bf Theorem A}-(a), in order to get non-trivial rank-$4$ extensions and so on.

\bigskip

We will limit ourselves to perform extensions of the {\em sporadic} bundle $[\cF_3] \in {\rm Ext}^1(L_2, \mathcal U_2)$ general by means of an Ulrich line bundle $A$ chosen among e.g. {\em non-sporadic} pairs $(L_1, L_2)$. On the other hand, as observed in \S\;\ref{Ulrich higher rk  vb}, if  in the next step  we considered further extensions ${\rm Ext}^1 (L_2, \cF_3)$, taking into account the associated coboundary map it is easy to see that the dimension of such an extension space drops by one with respect to that of ${\rm Ext}^1 (L_2, \cU_2)$. Therefore, keeping $L_2$ fixed on the right side of the extensions,  after finitely many steps we would get only splitting bundles. To avoid this fact we proceed once again by taking {\em alternating extensions}, namely 
 \[0 \to \cU_2 \to \cF_3 \to L_2 \to 0,\;\;\; 0 \to \cF_3 \to \cG_4 \to L_1 \to 0,\; \ldots , 
  \]
  and so on, that is, defining
  \begin{equation} \label{eq:tr}
    \tau_r: =
    \begin{cases}
      1, & \mbox{if $r \geqslant 4$ is even}, \\
      2, & \mbox{if $r \geqslant 3$ is odd},
    \end{cases}
  \end{equation}
  we take successive $[\cG_{r}] \in \Ext^1(L_{\tau_{r}},\cG_{r-1})$, for all $r \geqslant 3$, defined by:
  \begin{equation}\label{eq:3}
0 \to \cG_{r-1} \to \cG_{r} \to L_{\tau_{r}} \to 0,
 \end{equation}where, for $r=3$, $\cG_{r-1} = \cG_2:= \cU_2$ and $\cG_r = \cG_3 = \cF_3$. In particular, from the fact that $c_1(\cU_2) = c_1(M_1)+ c_1(M_2)$ and from 
 \eqref{eq:3}, one gets $c_1(\cG_r)  =   c_1(M_1)+ c_1(M_2) + \lfloor \frac{r-1}{2}\rfloor \, c_1(L_2) + \lfloor \frac{r-2}{2}\rfloor \, c_1(L_1)$, namely
{\footnotesize
\begin{eqnarray*}
c_1(\cG_r) = 2 \xi + \varphi^*\Oc_{\FF_0}(4,6t-2) + 
\left(  \lfloor \frac{r-1}{2}\rfloor \right) (\xi + \varphi^*\Oc_{\FF_0}(-1,2t-1)) + 
\left(  \lfloor \frac{r-2}{2}\rfloor \right) (\xi + \varphi^*\Oc_{\FF_0}(2,-1)),
\end{eqnarray*}} which reads
\begin{equation} \label{eq:c1rcasob}
    c_1(\cG_r): =
    \begin{cases} 
      r \xi +\varphi^*\Oc_{\FF_0}\left(\frac{(r-3)}{2},(r-1)(t-1) +1 \right), & \mbox{if $r$ is odd}, \\
      r \xi + \varphi^*\Oc_{\FF_0}\left(\frac{r}{2}, (r-2)(t-1)\right), & \mbox{if $r$ is even},
    \end{cases}
  \end{equation} (similar formulas can be determined for the other Chern classes).

Applying similar computations as in Lemma \ref{lemma:1}, with $M_1$ and $M_2$ replaced by $L_1$ and $L_2$, we deduce as in Corollary \ref{cor:Corollario al Lemma 4.2} that there exist rank-$r$ vector bundles $\cG_r$, which are Ulrich w.r.t. $\xi$, with $c_1$ as in \eqref{eq:c1rcasob}, of slope $\mu(\cG_r) = 13t-3$ and which arise as non-trivial extensions as in \eqref{eq:3}. Applying then similar strategies as in Lemma \ref{lemma:2}, we find recursive formulas:

\begin{equation}\label{chiGr}
\chi(\cG_r \otimes \cG_r^{\vee}) = 
\chi(L^{\vee}_{\tau_{r}} \otimes \cG_{r-1}) + \chi(\cG_{r-1} \otimes \cG_{r-1}^{\vee}) 
+ \chi(\mathcal O_X) + \chi(\cG_{r-1}^{\vee} \otimes L_{\tau_{r}}). 
\end{equation}

By induction, we can argue as in Lemma \ref{lem:new} and Corollary \ref{cor:new} to get that, starting from $[\cU_{r-1}] \in \mathcal M(r-1)$ general, for any $r \geqslant 2$, corresponding to a rank-$(r-1)$ vector bundle, which is Ulrich w.r.t. $\xi$ and slope-stable, of slope $\mu(\cU_{r-1}) = 13t-3$, whose first Chern class is as in \eqref{eq:c1rcasob}, then $[\cF_r] \in {\rm Ext}^1(L_{\tau_{r}}, \cU_{r-1})$ general is a rank-$r$ vector bundle, which is simple, so indecomposable, Ulrich w.r.t. $\xi$, with first Chern class as in \eqref{eq:c1rcasob}, of the same slope and with $h^j(\cF_r \otimes \cF^{\vee}_r) = 0$, for $2 \leqslant j \leqslant 3$. 
 
 So, from Proposition \ref{casanellas-hartshorne}, $\cF_r$ sits in a smooth modular family $\mathcal M(r)$, whose general element $\cU_r$ is, reasoning as in Lemma \ref{lemma:genUr}, Ulrich w.r.t. $\xi$, of rank $r$, with slope w.r.t. $\xi$ given by $\mu:= \mu(\cU_r) = 13t - 3$, with  first Chern class as in \eqref{eq:c1rcasob}, with  $h^j(\cU_r \otimes \cU_r^{\vee})=0$, for $j =2,3$ (by semi-continuity on $\mathcal M(r)$)  and with $\chi(\cU_r \otimes \cU_r^{\vee})$ as in \eqref{chiGr} (because 
 $\cG_r$ and $\cU_r$ have same Chern classes). Using same reasoning as in Lemmas \ref{lemma:uniquedest} and \ref{lemma:dimU}, one can show that the modular family $\mathcal M(r)$ is generically smooth, of dimension 
 $$h^1(\cU_r \otimes \cU_r^{\vee}) = 1 - \chi(\cG_r \otimes \cG_r^{\vee}),$$ as it follows from the vanishings $h^j(\cU_r \otimes \cU_r^{\vee})=0$, for $j =2,3$, and $h^0(\cU_r \otimes \cU_r^{\vee})=1$, from simplicity of $\cU_r$. Thus, one can conclude similarly as in Proposition  \ref{prop:slopstab}. 
 
 Therefore, to conclude the proof of {\bf Theorem C-mixed cases}, one is reduced to computing \eqref{chiGr}. Applying similar strategies as in 
Lemma \ref{lemma:2}, from \eqref{chiGr} one gets: 
 $$\chi(\cG_r \otimes \cG_r^{\vee})= 
\begin{cases}
  \scriptstyle    \frac{r^2 (5 - 10t) - 4r +  8t+4}{4}, & \mbox{if $r$ is even}, \\
		\scriptstyle	\frac{2r^2(2-t) - (14t +11) r + (24 t + 15)}{2}, & \mbox{if $r$ is odd}
    \end{cases}.$$Therefore, $\dim(\mathcal M(r)) =  h^1(\cU_r \otimes \cU_r^{\vee}) = 1 - \chi(\cG_r \otimes \cG_r^{\vee})$ is as stated in {\bf Theorem C-mixed cases}.

%%%%%
  %
  % BIBLIOGRAFIA
  %
  %%%%%%%%

\end{document}